\documentclass{birkjour}
\usepackage{amsmath, amsthm, amssymb, mathrsfs}
\usepackage{varioref}
 \newtheorem{thm}{Theorem}[section]
 \newtheorem{cor}[thm]{Corollary}
 \newtheorem{lem}[thm]{Lemma}
 \newtheorem{prop}[thm]{Proposition}
 \theoremstyle{definition}
 \newtheorem{defin}[thm]{Definition}
 \theoremstyle{remark}
 \newtheorem{rem}[thm]{Remark}
 
 \numberwithin{equation}{section}

\newcommand{\n}{\|}
\newcommand{\len}{\left\|}
\newcommand{\pn}{\right\|}
\newcommand{\la}{\left\langle}
\newcommand{\ra}{\right\rangle}
\newcommand{\lee}{\left(}
\newcommand{\p}{\right)}

\newcommand{\inv}[1]{\frac{1}{ #1 }}
\newcommand{\maxsym}{\vee}
\newcommand{\minsym}{\wedge}
\newcommand{\E}{\mathbb{E}}
\newcommand{\R}{\mathbb{R}}

\renewcommand{\P}{\mathbb{P}}

\newcommand{\F}{\mathcal{F}}
\newcommand{\FF}{\mathbb{F}}
\newcommand{\eqtwo}[4]{\left\{\begin{array}{ll} #1 & #2  \\ #3 & #4 \end{array}\right.}

\newcommand{\Ep}[1]{\mathcal{E}_{#1}}
\newcommand{\Ca}[1]{\mathcal{#1}}
\newcommand{\Bo}[1]{\mathbb{#1}}
\newcommand{\norm}[1]{\left\|#1\right\|}

\newcommand{\vecfour}[4]{\left[\begin{array}{ll} #1 & #2  \\ #3 & #4 \end{array}\right]}

\begin{document}

%
%
%
%
%
%
%
%
%

\title[On the equivalence of solutions for an evolution equation]{On the equivalence of solutions for  a class  of stochastic evolution equations  in a Banach space}

\author[M. G\'orajski]{Mariusz G\'orajski}

\address{%
Department of Econometrics,\\ Faculty of Economics and Sociology, \\University of  \L \'od\'z\\
Rewolucji 1905 r. No. 41\\
90-214 \L \'od\'z, Poland
}

\email{mariuszg@math.uni.lodz.pl}



\date{October 14 , 2013}

\begin{abstract}
We study a class of stochastic evolution equations in a Banach space $E$ driven by cylindrical Wiener process.
 Three different analytical concepts of solutions: generalised strong, weak and mild are defined and the conditions under which they are equivalent are given. We apply this result to prove existence, uniqueness and continuity of weak solutions to stochastic delay evolution equations. We also consider two examples of these equations in non-reflexive Banach spaces: a stochastic transport equation with delay and a stochastic delay McKendrick equation.
\end{abstract}

\maketitle
\renewcommand\theenumi{\roman{enumi}}
\renewcommand\labelenumi{(\theenumi)}
\section*{Introduction}
Let $E$ be a Banach space and let $H$ be a separable Hilbert space. In a given probability basis $\left((\Omega,\F, \mathbb{F},\P),W_H\right)$, i.e. $(\Omega,\F,\P)$ is a complete probability space and $W_H$ is an $H$-cylindrical Wiener process  with respect to a complete filtration $ \mathbb{F}=(\F_t)_{t\geq 0}$ on $(\Omega,\F,\P)$,
 consider the stochastic evolution equation:
\begin{align}\label{SCP}
\left\{ \begin{array}{rll} dY(t)&=AY(t)dt +F(Y(t))dt+ G(Y(t))dW_H(t), &t\geq 0;\\
Y(0)&=Y_0, \end{array}\right. \tag{SCP}
\end{align}
for initial condition $Y_0\in L^0((\Omega,\F_0);E)$. Here $(A,D(A))$ generates a strongly continuous semigroup  $(T(t))_{t\geq 0}$ on a Banach space $\tilde{E}$ such that $D(A)\subset E\subset\tilde{E}$ are continuous and dense embeddings, and nonlinearities $F:D(F)\subset E\rightarrow \tilde{E}$ and $G:D(G)\subset E\rightarrow \Ca{L} (H,\tilde{E})$ are strongly measurable mappings with some regularity properties which we make precise in Section \ref{s:prelim}. 

In the definition of a solution to the stochastic equation \eqref{SCP} one can either fix the probability
basis $\left((\Omega,\F, \mathbb{F},\P),W_H\right)$ on which the process $Y$ lives in advance or make this to be part of a solution.
In the former case the solution $Y$ is usually called a \textit{stochastically strong solution}, whereas in the latter case $Y$ is a martingale or
a weak solution in the probabilistic sense (\textit{stochastically weak solution}). Here we only consider stochastically strong solutions of \eqref{SCP}. For the notion of martingale solution of \eqref{SCP} in a Banach space see \cite{Kunze}, where inter alia the equivalence between weak solutions in the probabilistic sense and local martingale problem is established. 
 
We recall three analytical concepts of solutions to \eqref{SCP}.   

\begin{defin}\label{d:mildsol}
A strongly measurable, $\mathbb{F}$- adapted process $Y$ is called a \emph{mild solution} to \eqref{SCP} if for
all $t> 0$ we have
\begin{enumerate}
    \item there exists the Bochner integral $\int_{0}^{t}T(t-r)F(Y(r))dr$ a.s.;
		\item there exists the stochastic integral $\int_{0}^{t}T(t-r)G(Y(r))dW_H(r)$;
		\item for almost all $\omega$ \begin{align}\label{voc2}
Y(t) &= T(t)Y_0 + \int_{0}^{t}T(t-r)F(Y(r))dr+\int_{0}^{t}T(t-r)G(Y(r))dW_H(r).
\end{align}
\end{enumerate}  
  \end{defin}
We introduce the following definition of weak solution to \eqref{SCP} which is slightly more general than the one considered in \cite{vanNeervenVeraarWeis2} and \cite{DaPratoZabczyk}.

\begin{defin}\label{d:weaksolSCP}
A strongly measurable, $\mathbb{F}$-adapted process $Y$ is called a \emph{weak solution} to \eqref{SCP} if $Y$ is a.s.\ locally Bochner integrable and for
all $t> 0$ and $x^*\in D(A^\odot)$:

\begin{enumerate}
\item $\la F(Y),x^*\ra$ is integrable on $[0,t]$  a.s.;
\item $G^*(Y)x^*$ is stochastically integrable on $[0,t]$;
\item  for almost all $\omega$ 
\begin{align*}
\la Y(t)-Y_0 ,x^*\ra= \int_{0}^{t}\la Y(s),A^\odot x^* \ra ds &+\int_0^t\la F(Y(s)), x^*\ra ds\\&\quad +\int_{0}^{t} G^*(Y(s))x^* dW_H(s).
\end{align*}
\end{enumerate}

\end{defin}

In the following interpretation of solution to \eqref{SCP} we use the theory of stochastic integration in a Banach space as given in \cite{vanNeervenVeraarWeis}. 
\begin{defin}\label{d:strongsolSCP}
A strongly measurable, $\mathbb{F}$-adapted process $Y$ is called a \emph{generalized strong solution} to \eqref{SCP} if $Y$ is a.s.\ locally Bochner integrable and for
all $t> 0$:
\begin{enumerate}
\item $\int_{0}^{t} Y(s)ds \in D(A)$ a.s.,
\item $F(Y)$ is Bochner integrable in $[0,t]$ a.s.,
\item $G(Y)$ is stochastically integrable on $[0,t]$,
\end{enumerate}
and
$$Y(t) - Y_0 = A \int_{0}^{t} Y(s)ds + \int_{0}^{t} F(Y(s)) ds+\int_{0}^{t} G(Y(s)) dW_H(s)\quad a.s.$$
\end{defin}
A process $Y$ satisfying Definition \ref{d:strongsolSCP} is called a \textit{analytically strong solution} to \eqref{SCP} if  in addition $Y(t)\in D(A)$ a.s. for all $t>0$ and $AY$ is locally Bochner integrable (see \cite{DaPratoZabczyk}). The additional condition in the definition of analytically strong solution is not appropriate
for stochastic delay equations (see Remark 4.10 in \cite{coxgorajski})  which we consider in Section \ref{s:SDEaddtive}, thus in this paper we do not focus on this concept of solution.\par
The equivalence of these three interpretations of solution to \eqref{SCP} in Hilbert space has been proved by Chojnowska-Michalik in \cite{chojnowskaMichalik_SDEHilbert} (see also \cite[Theorem 6.5]{DaPratoZabczyk} and \cite[Theorem 9.15]{PeszZab}). 
 For a linear \eqref{SCP} with additive noise in a Banach space the equivalence of weak and mild solution is given in \cite{BrzezniakNeerven_stochconv}, while in~\cite[Theorems 8.6 and 8.10]{vanNeervenISEM} one can find the proof of equivalence of these three concepts of solution. In \cite{veraar_thesis}, the author considers mild, variational and weak solutions of  non-autonomous stochastic Cauchy problems in a \textsc{umd}$^{-}$ Banach  space.  Applying the stochastic Fubini theorem he proves that mild and variational interpretations are identical. Moreover, only for reflexive Banach spaces, using Ito's formula, it is shown in \cite{veraar_thesis} that weak and variational concepts are equivalent. In \cite{coxgorajski}, the authors consider linear stochastic Cauchy problems in a \textsc{umd}$^-$ Banach space and formulate sufficient conditions for equivalence of mild and generalised strong solutions of \eqref{SCP} (see Theorem 3.2 in~\cite{coxgorajski}). In the weak probabilistic setting and in a \textsc{umd} Banach space, assuming continuity of paths in the definitions of solutions and using localization and It\^o's formula in the proofs, the equivalence between analytically weak and mild solutions  is shown by Kunze in \cite[Section 6]{Kunze}. \par

In this paper we prove that the equivalence of Definitions \ref{d:mildsol}-\ref{d:strongsolSCP} is also valid in  \textsc{umd}$^-$ Banach spaces. Sections \ref{s:weakmild} and \ref{s:threesol} show that in the fixed probability basis and without the assumption on the paths continuity three above-defined  concepts of solutions to \eqref{SCP} are equivalent (Theorems \ref{t:varcons}, \ref{t:varcons2}). These theorems are used in \cite[Theorem 4.8]{coxgorajski} and in \cite[Theorems 3.2, 3.6]{GorajskiSDE} to prove Markovian representation of stochastic delay equations in $E\times L^p(-1,0;E)$ for some $p\geq 1$, where $E$ is a type 2 \textsc{umd} Banach space. In Section \ref{s:SDEaddtive} we apply the equivalence from Theorem \ref{t:varcons} to prove the existence, uniqueness and continuity of weak solutions to a class of stochastic delay evolution equations in an arbitrary separable Banach space $E$ (see Proposition \ref{p:}).  \par 
 
It is worth mentioning that it turns out that for stochastic evolution equations with non-additive noise in a \textsc{umd} Banach space which are not type 2 it is convenient to analyse a concept of mild $E_\eta$-solution of \eqref{SCP}. This interpretation is more general than these considered in the article.  The existence, uniqueness and H\"{o}lder regularity results of mild $E_{\eta}$-solution to \eqref{SCP} with $A$ being an analytic generator has been proved in  \cite{vanNeervenVeraarWeis_SEEinUMD}.
Since the delay semigroup is not an analytic semigroup, we can not use these results in Section \ref{s:SDEaddtive}.

In the next section, mainly based on \cite{vanNeervenVeraarWeis}, we present sufficient conditions for the existence of stochastic integral in a \textsc{umd}$^-$ Banach space, and some preliminary lemmas which will be useful in the sequel. 

\section{Preliminaries}
\label{s:prelim}
Here and subsequently, $E, \tilde{E}$ stand for real Banach spaces, $H$ denotes a real separable Hilbert space and $W_H$ is an $H$-cylindrical Wiener process on a given probability space $(\Omega,\F, \mathbb{F},\P)$. The following hypothesis will be assumed. 

\renewcommand\labelenumi{(H0)}
\labelformat{enumi}{(H0)}

\begin{enumerate}
    \item \label{h0} $A:D(A)\subset
\tilde{E}\rightarrow \tilde{E}$ is a generator of a $C_0$-semigroup $(T(t))_{t\geq0}$ on a Banach space  $\tilde{E}$ such that $T(t)\in\Ca{L}(E)$ for all $t>0$ and $D(A)\subset E\subset\tilde{E}$ are continuous and dense embeddings.
\end{enumerate}
In a typical example in which hypothesis \ref{h0} is satisfied, $E=\tilde{E}$ and $(A, D(A))$ generates a strongly continuous semigroup $(T(t))_{t\geq0}$ on $E$ (see also Example 3.2 in \cite{Kunze} ).

In the case where $\tilde{E}$ is not reflexive the adjoint semigroup $(T^{*}(t))_{t\geq 0}$ is not necessary strongly continuous (cf. \cite{EngNag}). However \textit{sun dual semigroup} $(T^{\odot}(t))_{t\geq 0}$ defined as subspace semigroup by $T^{\odot}(t)=T^{*}(t)_{|\tilde{E}^{\odot}}$ defined on   $\tilde{E}^{\odot}=\overline{D(A^*)}$ is strongly continuous (see. 2.6 in \cite{EngNag} and Chapter 1 in \cite{vanNeervenadjoint}). A generator $(A^{\odot},D(A^{\odot}))$ of the sun dual semigroup is given by $A^\odot=A^{*}_{|D(A^{\odot})}$ and $D(A^{\odot})=\{x^*\in D(A^*): A^*x^*\in \tilde{E}^{\odot}\}$.\par 
\begin{lem}\label{l:*weakdense}
For all $n\geq 1$ the set $D((A^\odot)^n)$ separates points in $\tilde{E}$.  
\end{lem} 
\begin{proof}
Let $n\geq 1$. By strong continuity $D((A^{\odot})^n)$ is dense in $\tilde{E}^\odot$ (see Proposition 1.8 in \cite{EngNag}), hence it is also $*$-weak dense in $\tilde{E}^\odot$. By  Theorem 1.3.1 in \cite{vanNeervenadjoint} it follows that $\tilde{E}^\odot$ is $*$-weak dense in $\tilde{E}^*$. Thus $D((A^\odot)^n)$ is $*$-weak dense in $E^*$. The last property of $D((A^\odot)^n)$ gives the assertion of the lemma.   
\end{proof}
 In the rest of this paper we assume the hypotheses.  
\renewcommand\theenumi {\Alph{enumi}}
\renewcommand\labelenumi{(H\theenumi)}
\labelformat{enumi}{(H#1)}

\begin{enumerate}
    \item \label{F} $F:D(F)\subset E\rightarrow \tilde{E}$ is strongly measurable, $D(F)$ is dense in $E$ and there exists
     $a\in L^1_{loc}(0,\infty)$ such that for all $t>0$ and $x,y\in D(F)$ we have  $T(t)F(x)\in E$ and
\begin{align*}
&\|T(t)F(x)\|_{E}\leq a(t)(1+\|x\|_E),\\
&\|T(t)\left(F(x)-F(y)\right)\|_{E}\leq a(t)\|x-y\|_E,
\end{align*}
\item \label{G} $G:D(G)\subset E\rightarrow \Ca{L} (H,\tilde{E})$ is $H$-strongly measurable, $D(G)$ is dense in $E$ and there exists $b\in L^2_{loc}(0,\infty)$ such that for all $t>0$ and $x,y\in D(G)$ we have $T(t)G(x)\in \Ca{L}(H,E)$   and
\begin{align*}
&\|T(t)G(x)\|_{\Ca{L}(H,E)}\leq b(t)(1+\|x\|_E),\\
&\|T(t)\left(G(x)-G(y)\right)\|_{\Ca{L}(H,E)}\leq b(t)\|x-y\|_E.
\end{align*} 
\end{enumerate}
\begin{lem}\label{l:FGweak}
If \ref{h0}, \ref{F} and \ref{G} hold, then for all $x^*\in D(A^*)$ there exist constants $C_1(x^*),\ C_2(x^*)>0$ such that  
\begin{align*}
&|\la F(x),x^*\ra| \leq C_1(x^*)(1+\|x\|_E),\\
&|\la F(x)-F(y),x^*\ra| \leq C_1(x^*)\|x-y\|_E,
\end{align*} 
and
\begin{align*}
&\norm{G^*(x)x^*}_H\leq C_2(x^*)(1+\|x\|_E),\\
&\norm{(G^*(x)-G^*(y))x^*}_H\leq C_2(x^*)\|x-y\|_E
\end{align*} 
for all $x,y\in D(G)$.

\end{lem}
\begin{proof}
In the case where $\tilde{E}=E$ is a Hilbert space see Lemma 9.13 in \cite{PeszZab}. If \ref{h0} holds, then we can repeat the reasoning from the proof of Lemma 9.13. Indeed, by strong continuity of $(T(t))_{t\geq 0}$ and \ref{F} there exists $\lambda>0$ such that for all $x\in D(F)$ we have the following inequality 
\begin{align}
&\notag\len(\lambda I-A)^{-1}F(x)\pn_{\tilde{E}}\leq \int_0^\infty e^{-\lambda t}\len T(t)F(x)\pn_{\tilde{E}}dt\\
&\hspace{1.5cm}\leq K_E\left( \int_0^1a(t)dt+a(1)\int_1^\infty e^{-\lambda t}\len T(t-1)\pn_{\Ca{L}(\tilde{E})}dt\right)\left(1+\len x\pn_E\right)\label{Fweak}\\&\hspace{1.5cm}= C_1(1+\len x\pn_E),\notag
\end{align} 
where $K_E>0$ is a norm of continuous, linear embedding $E\to \tilde{E}$. 
Since $(A,D(A))$ is closed and densely defined on $\tilde{E}$, $(\lambda I-A^*)^{-1}=((\lambda I-A)^{-1})^*$ (cf. B.11-12 in \cite{EngNag}). Hence $D(A^*)=D((\lambda I-A)^*)=(\lambda I-A^*)^{-1}(\tilde{E}^*)$. Then for all $x^*\in D(A^*)$ there exist $y^*\in \tilde{E}^*$ such that $x^*=(\lambda I-A^*)^{-1}y^*$. By \eqref{Fweak} we obtain 
$$|\la F(x),x^*\ra|=\left|\la (\lambda I-A)^{-1}F(x),y^*\ra\right|\leq \n y^*\n_{\tilde{E}^*}C_1(1+\n x\n_E).$$ 
The inequalities for   $G^*$ may be handled in much the same way. 
Applying \ref{G} we conclude that  for every $x\in D(G)$ and all $h\in H$  \begin{align}\label{Fweak2}
\n(\lambda I-A)^{-1}G(x)h\n_{\tilde{E}}&\leq \int_0^\infty e^{-\lambda t}\n T(t)G(x)\n_{\Ca{L}(H,\tilde{E})}dt\n h\n_H\\&\leq C_2(1+\n x\n_E)\n h\n_H,\notag
\end{align} 
where $C_2=K_E\left(\int_0^1b(t)dt+b(1)\int_1^\infty e^{-\lambda t}\n T(t-1)\n_{\Ca{L}(\tilde{E})}dt\right)$.
Hence we obtain for all $x^*=(\lambda I-A^*)^{-1}y^*\in D(A^*)$
\begin{align}\notag
\n G^*(x)x^*\n_H&=\sup_{h\in H,\n h\n\leq 1}\left|[h,G^*(x)x^*]_H\right|=\sup_{h\in H,\n h\n\leq 1}\left|\la G(x)h,x^*\ra\right|\\&\leq \n y^*\n_{\tilde{E}^*}C_2(1+\n x\n_E). \label{Fweak3}
\end{align}
\end{proof} 

In the sequel we use the theory of stochastic integral for $\Ca{L}(H,E)$-valued process as introduced in \cite{vanNeervenVeraarWeis}. For a Banach space with \textsc{umd} property  one may characterise stochastic integrability in terms of $\gamma$-radonifying norm. \textsc{umd} property stands for Unconditional Martingale Difference property and it requires that all $E$-valued, $L^q(\Omega;E)$-convergent  sequences of martingale differences are unconditionally convergent (see \cite{Garling_rmte} and \cite{vanNeervenVeraarWeis}). Throughout the paper, 
 $\gamma(H,E)$ stand for the space of $\gamma$-radonifying linear operators from $H$ to $E$. The space $\gamma(H,E)$ is defined to be the closure of the finite rank operators under the norm   
$$ \len R \pn_{\gamma(H,E)}^2 := \sup_{(h_j)_{j=1}^k} \E\ \len \sum_{j=1}^k \gamma_j R h_j \pn_E^2,$$
 where the supremum is taken over all finite orthonormal systems 
$h=(h_j)_{j=1}^k$ in $H$ and $(\gamma_j)_{j\geq 1}$
is a sequence of independent standard Gaussian random variables. Hence $\gamma(H,E)$  is a separable Banach space.

A $H$-strongly measurable, adapted process $\Psi:[0,t]\times\Omega\to \Ca{L}(H,E)$ such that $\Psi^*x^*\in L^2(0,t;H)$ holds a.s. for all $x\in E^*$  is stochastically integrable with respect to the cylindrical Wiener process $W_H$ if and only if  $\Psi$ represents $\gamma(L^2(0,t;H),E)$-valued random variable $R_{\Psi}$ given by 
\begin{align}\label{R}
\la R_{\Psi}f, x^*\ra=\int_0^t\la\Psi(s)f(s),x^*\ra ds \quad \textrm{a.s},
\end{align}
for every $f\in L^2(0,t;H)$ and for all $x^*\in E^*$ (see \cite[Theorem 5.9]{vanNeervenVeraarWeis} ). For the sake of simplicity we shall say then that the process $\Psi$ is in $\gamma(L^2(0,t;H);E)$ a.s. (see also Lemma 2.5, 2.7 and Remark 2.8 in \cite{vanNeervenVeraarWeis}).
If one wants a stochastic integral $\int_0^t\Psi dW_H$ to be in $L^q(\Omega;E)$ for some $q>1$, then assuming that the process $\Psi$ is scalarly in $L^q(\Omega,L^2(0,t;H))$ (i.e. $\Psi^*x^*\in L^q(\Omega;L^2(0,t;H))$) the $L^q$- stochastic integrability of $\Psi$ can be characterised by the existence of a random variable $\xi\in L^q(\Omega;E)$ such that for all $x^*\in E^*$ 
\begin{align}\label{xi}
\la\xi, x^*\ra=\int_0^t \Psi^*(s)x^*dW_H(s) \text{ in } L^q(\Omega).
\end{align}
 Using decoupling inequalities (see \cite{coxVeraar}) one can prove the Burkholder-Gundy-Davies type inequality :
\begin{align}\label{BDG} \E\ \sup_{s\in[0,t]}\len\int_0^s\Psi(u)dW_H \pn^q_E\eqsim_q\E\ \len R_\psi\pn_{\gamma(L^2(0,t;H),E)}^q
\end{align}
for all $q>0$\footnote{For reals $A,B$ we use the notation $A\lesssim_q B$ to express the fact that there exists a constant $C>0$, depending on $q$, such that $A\leq CB$. We
write $A\eqsim_q B$ if $A\lesssim_q B\lesssim_q A$.}. 

In \cite{Garling_rmte} it is shown that \textsc{umd} property can be characterised in terms of two properties: \textsc{umd}$^-$ and \textsc{umd}$^+$. 
\begin{defin}\label{d:UMD-}
 A Banach space $E$ has \textsc{umd}$^-$ property, if for every $1<q<\infty$ there exists $\beta^-_q>0$ such that for all $E$-valued sequences of $L^q$-martingale differences $(d_n)_{n=1}^N$ and for all Rademacher sequences $(r_n)_{n=1}^N$ independent from $(d_n)_{n=1}^N$ we have the following inequality
\begin{align}\tag{\textsc{umd}$^-$}\label{umd-}\E \ \norm{\sum_{n=1}^{N} d_n}_E^q\leq \beta_q^- \E \ \norm{\sum_{n=1}^{N} r_nd_n}_E^q. 
\end{align}
\end{defin}
 A Banach space $E$ has \textsc{umd}$^+$ property, if the reverse inequality to  \eqref{umd-} holds. Recall that class of \textsc{umd} Banach spaces is in the class of reflexive spaces and includes  Hilbert spaces and $L^q$ spaces for $q\in(1,\infty)$. Moreover, class of \textsc{umd}$^-$ Banach spaces includes also non-reflexive $L^1$ spaces.\par
 To integrate processes with values in \textsc{umd}$^-$ one needs a weaken notion of stochastic integral. In a Banach space $E$ with \textsc{umd}$^-$ property the condition: $\Psi$ is in $\gamma(L^2(0,t;H),E)$ a.s. is sufficient for the process $\Psi$ to be stochastic integrable (cf. \cite[Proposition 3.4]{vanNeervenVeraar_Fub} and \cite[Section 8]{vanNeervenVeraarWeis_SEEinUMD}).   
Moreover,  if $\FF$ is the augmented Brownian filtration $\FF^{W_H}$ and $\Psi^*x^*\in L^q(\Omega;L^2(0,t;H))$, and \eqref{xi} holds for all $x^*\in E^*$, then it follows by the martingale representation theorem that $\Psi$ is $L^q$-stochastically integrable with respect to $W_H$. 

\section{Equivalence of weak and mild solutions}
\label{s:weakmild}
\begin{thm}\label{t:varcons}
In a \textsc{umd}$^-$  Banach space $E$, consider \eqref{SCP} with  hypotheses \ref{h0}, \ref{F} and \ref{G}. Let $Y$ be an $E$-valued strongly measurable, adapted process with almost surely locally Bochner square integrable trajectories. Assume that one of the following conditions holds
\renewcommand\theenumi {\roman{enumi}}
\renewcommand\labelenumi{(\theenumi)}
\begin{enumerate}
\item $E$ is a \textsc{umd} space or $\FF=\FF^{W_H}$, and $\sup_{s\in[0,t]} \E \n Y(s)\n^q_E<\infty$ for all $t\in(0,\infty)$ and some $q>1$;
\item for all $t> 0$ the process:
\begin{align}\label{G1}
u\mapsto T(t-u)G(Y(u))
\end{align}
 is in $\gamma(L^2(0,t;H),E)$ a.s.
\end{enumerate} 
Then $Y$ is a weak solution to \eqref{SCP} if and only if $Y$ is a mild solution to \eqref{SCP}. \\
Moreover, if there exists solution and the hypothesis (i) holds, then $u\mapsto T(t-u)G(Y(u))$ represents an element from $L^q(\Omega;\gamma(L^2(0,t;H),E))$ for all $t>0$.

\end{thm}
\renewcommand\theenumi{\roman{enumi}}
\renewcommand\labelenumi{(\theenumi)}

Before proving the theorem, we formulate some remarks which are the consequences of Lemma \ref{l:FGweak} and the properties of stochastic integral in Banach spaces. 
\begin{rem}\label{r:varcons}Fix $x^*\in D(A^*)$. Let $Y$ be a $E$-valued, strongly measurable adapted process with locally Bochner square integrable trajectories a.s. 
\begin{enumerate}
\item Condition \ref{F} and Lemma \ref{l:FGweak} implies that $E\ni x\mapsto \la F(x),x^*\ra \in \R{}$ is a Lipschitz-continuous function. Hence the condition (i) from the definition of weak solution to \eqref{SCP} is satisfied.
 
\item From \ref{G} and Lemma \ref{l:FGweak} it follows that $E\ni x\mapsto G^*(x)x^*\in H$ is a Lipschitz-continuous function. Hence the process  $G^*(Y)x^*$ is strongly measurable and adapted with locally square integrable trajectories a.s. In particular $G^*(Y)x^*$ is stochastically integrable on $[0,t]$ for all $t>0$.  

\item By \ref{F} and \ref{G} the mappings $E\ni x\mapsto T(s)F(x)\in E$, $E\ni x\mapsto T(s)G(x)\in\Ca{L}(H,E)$   are continuous functions, hence processes $$T(t-\cdot)F(Y(\cdot)),\quad T(t-\cdot)G(Y(\cdot))$$ are  adapted, strongly and $H$-strongly measurable, respectively. Moreover, the first process  has trajectories locally Bochner  integrable a.s. 
\end{enumerate} 
\end{rem}
\begin{proof}[Proof of Theorem \ref{t:varcons}]
We apply the stochastic Fubini theorem from \cite{vanNeervenVeraar_Fub} to obtain the key equations for the proof of Theorem \ref{t:varcons}: equations \eqref{r:stochFub1}, \eqref{r:stochFub2} and \eqref{r1} below. As the process $Y$ is assumed to be strongly measurable and adapted we may assume without loss of generality that $E$ is separable. Moreover, observe that since every adapted and measurable process with values in Polish
 space has a progressive version, we may assume that $Y$ is progressive. \par

\textbf{Step 1.}
Fix $x^*\in D(A^*)$ and $t>0$.  Consider the processes :
\begin{align*}
  \Psi^*_{1}x^*(s,u,\omega)=1_{[0,s]}(u)G^*(Y(u,\omega))T^*(t-s)x^*,\\
 \Psi^*_{2}x^*(s,u,\omega)=1_{[0,s]}(u)G^*(Y(u,\omega))T^*(s-u)x^*,
\end{align*}
which are formed from $\Ca{L}(H,E)$-valued  possesses  $\Psi_{1}$, $\Psi_{2}$ given by
$$\Psi_{1}=1_{[0,s]}(u)T(t-s)G(Y(u,\omega)), \Psi_{2}=1_{[0,s]}(u)T(s-u)G(Y(u,\omega)),$$
where $s,u\in[0,t]$.
By Remark \ref{r:varcons}.(iii) it follows that  $\Psi_{1}$, $\Psi_{2}$ are  $H$-strongly measurable. As $Y$ is assumed to be progressive we conclude that for all $s\in[0,t]$ and $h\in H$ the selections: $(\Psi_{1})_sh(u,\omega):=\Psi_{1}(s,u,\omega)h$, $(\Psi_{2})_sh:=\Psi_{2}(s,u,\omega)h$ are progressive. Hence by Proposition 2.2 in \cite{vanNeervenVeraarWeis} we obtain strong measurability of  $\Psi^*_{1}x^*$, $\Psi^*_{2}x^*$ and for all $s\in[0,t] $ progressive measurability of $\Psi^*_{1}x^*(s)$, $\Psi^*_{2}x^*(s)$. To apply the stochastic Fubini theorem (Theorem 3.5 in \cite{vanNeervenVeraar_Fub}) to $\Psi^*_{1}x^*$, $\Psi^*_{2}x^*$ it is sufficient to show that
$$\int_0^t\left(\int_0^t\n\Psi^*_{1}x^*\n_H^2 du\right)^\inv{2}ds,\int_0^t\left(\int_0^t\n\Psi^*_{2}x^*\n_H^2 du\right)^\inv{2}ds<\infty \quad \textrm{a.s.}$$
Indeed, using  Lemma \ref{l:FGweak}  we have the following estimate
\begin{align}\label{2}
\int_0^t\lee\int_0^t\n\Psi^*_2x^*\n_H^2 du\p^\inv{2}&ds=\int_0^t\lee\int_0^{s}\n G^*(Y(u))T^*(s-u)x^*\n_H^2 du\p^\inv{2}ds\\
&\leq \int_0^t\lee\int_0^{s} C^2(s-u)(1+\n Y(u)\n_E)^2 du\p^\inv{2}ds \quad \textrm{a.s.,} \notag
\end{align}
where $C(s-u)=\n y^*_{s-u}\n_{E^*} C_2$ is the constant occurring in inequality  \eqref{Fweak3} for $y^*_{s-u}\in \tilde{E}^*$ such that $ (\lambda I -A^*)^{-1}y^*_{s-u}=T^*(s-u)x^*$.  Since $T^*(s-u)x^*\in D(A^*)$ and $A^*T^*(s-u)x^*=T^*(s-u)A^*x^*$ (see Proposition 1.2.1 in \cite{vanNeervenadjoint}), we have
\begin{align}\label{3}
C(s-u)&= C_2\n T^*(s-u)(\lambda I -A^*)x^*\n_{E^*}\\&\leq C_2 M(t)\n(\lambda I -A^*)x^*\n_{E^*}, \notag 
\end{align}
where $M(t)=\sup_{s\in[0,t]}\n T(s)\n_{\Ca{L}(E^*)}$.
Hence combining \eqref{3} and \eqref{2} we obtain, almost surely,
$$\int_0^t\lee\int_0^t\n\Psi^*_2x^*\n_H^2 du\p^\inv{2}ds\leq tC_2 M(t)\n(\lambda I -A^*)x^*\n_{\tilde{E}^*}(\sqrt{t}+\n Y\n_{L^2(0,t;E)}).$$
In the similar way we get 
$$\int_0^t\lee\int_0^t\n\Psi^*_1x^*\n_H^2 du\p ^\inv{2}ds\leq tC_2 M(t)\n(\lambda I -A^*)x^*\n_{E^*}(\sqrt{t}+\n Y\n_{L^2(0,t;E)}).$$
Thus from stochastic Fubini's theorem it follows that
\begin{align}\label{r:stochFub1}
&\int_0^t \int_0^{s} G^*(Y(u))T^*(t-s)x^* dW_H(u)ds =\int_0^{t} \int_0^t\Psi^*_1x^*dsdW_H(u),\\  
&\int_0^t \int_0^{s} G^*(Y(u))T^*(s-u)x^* dW_H(u)ds =\int_0^{t} \int_0^t \Psi^*_2x^*dsdW_H(u).   \label{r:stochFub2}
\end{align}
Moreover, notice that for all $u\in[0,t]$ we obtain
\begin{align} 
\label{r:intPsi1}
\int_0^t\Psi^*_1x^*ds=\int_0^t\Psi^*_2x^*ds=\int_0^{t-u}G^*(Y(u,\omega))T^*(s)x^*ds \quad \textrm{a.s.}
\end{align} 
By \eqref{r:intPsi1} and strong continuity of $(T^\odot(t))_{t\geq 0}$ it follows that for all $x^*\in D(A^\odot)$ we have, almost surely, 
\begin{align}\notag 
\int_0^{t}\int_0^t \Psi^*_1A^\odot x^*dsdW_H(u)&= \int_0^{t}\int_0^t \Psi^*_2A^\odot x^*dsdW_H(u)
\\&=\int_{0}^{t} \int_{0}^{t-u}G^*(Y(u))T^\odot(s)A^\odot x^*dsdW_H(u)\label{r1} \\ 
&=\int_{0}^{t}  G^*(Y(u))\int_{0}^{t-u} T^\odot (s)A^\odot x^*ds dW_H(u)\notag\\ 
&=\int_{0}^{t} G^*(Y(u))\lee T^\odot (t-u) x^*-x^*\p dW_H(u).\notag 
\end{align}

\textbf{Step 2.} Let us suppose that $Y$ is a weak solution to \eqref{SCP}, we prove that $u\mapsto T(t-u)G(Y(u))$
is in $L^0(\Omega;\gamma(L^2(0,t;H),E))$ and
 \eqref{voc2} holds. From \eqref{r1} and \eqref{r:stochFub1} and by the definition of a weak solution, we conclude that for all $x^*\in D(A^\odot)$ and $t>0$ one has, almost surely,
\begin{align}\notag
& \la Y(t)-Y_0,x^*\ra- \int_0^{t}\la Y(s), A^\odot x^* \ra ds -\int_0^{t}\la F(Y(s)),x^*\ra ds\\&=\int_0^{t}G^*(Y(u))x^*dW(u) \notag  \\&\stackrel{\eqref{r1}}{=}\int_{0}^{t} G^*(Y(u))T^\odot(t-u)x^*dW_H(u)- \int_{0}^{t}\int_{0}^{t} \Psi^*_1A^\odot x^*dsdW_H(u) \notag \\
&\stackrel{\eqref{r:stochFub1}}{=} \int_{0}^{t} G^*(Y(u))T^\odot(t-u)x^*dW_H(u) \label{r3}\\&\hspace{4cm}- \int_{0}^{t}\int_{0}^{s} G^*(Y(u))T^\odot(t-s)A^\odot x^*dW_H(u)ds. \notag 
\end{align}
Assuming that $y^*=A^\odot x^*\in D(A^\odot)$  and  using the definition of a weak solution again, we can write the last term in \eqref{r3} as follows
\begin{align}
 \notag \int_{0}^{t}&\int_{0}^{s} G^*(Y(u))T^\odot(t-s)A^\odot x^*dW_H(u)ds=\int_0^t\Big[\la Y(s)-Y_0, T^\odot(t-s)y^*\ra \\
&- \int_0^s\la F(Y(u)), T^\odot(t-s)y^*\ra du- \int_0^s\la Y(u),A^\odot T^\odot(t-s)y^*\ra du\Big]ds \label{r4} \\
&=-\int_0^t\la Y_0,T^\odot(t-s)y^*\ra ds+\int_0^t\la Y(s),y^*\ra ds\notag\\&-\int_0^t\int_u^t\la F(Y(u)),T^\odot(t-s)y^*\ra dsdu \text{ a.s.,} \notag 
\end{align}
where the first equality follows from condition (iii) of the definition of a weak solution to \eqref{SCP} and in the last equality we use strong continuity of  $(T^\odot(t))_{t\geq 0}$ and Fubini's theorem.
Applying \eqref{r4} into \eqref{r3} we obtain, almost surely,
\begin{align*}
\la Y(t),x^*\ra&-\la Y_0,x^*\ra-\la \int_0^t Y(s)ds,A^\odot x^* \ra -\int_0^t\la F(Y(s)),  x^*\ra ds\\
&=\int_{0}^{t} G^*(Y(u))T^\odot(t-u)x^*dW_H(u)+\int_0^t\la Y_0, T^\odot(t-s)A^\odot x^*\ra ds
\\
&\quad-\int_0^t\la Y(s),  A^\odot x^*\ra ds+\int_0^t\int_u^t\la F(Y(u)), T^\odot(t-s)A^\odot x^*\ra dsdu\\
&=\int_{0}^{t} G^*(Y(u))T^\odot (t-u)x^*dW_H(u)+\la Y_0, T^\odot(t)x^*\ra-\la Y_0, x^*\ra
\\
&\quad-\int_0^t\la Y(s), A^\odot x^*\ra ds+\int_0^t \la F(Y(u)), T^\odot(t-u)x^*\ra du\\
&\quad-\int_0^t \la F(Y(u)), x^* \ra du, 
\end{align*}
By Remark \ref{r:varcons}.(iii) the process  $T(t-\cdot)F(Y(\cdot))$ has, almost surely, trajectories Bochner integrable on $[0,t]$.
 Hence for all $x^*\in D((A^\odot)^2)$ one has, almost surely,
\begin{align}
\notag\la Y(t), x^*\ra &=\la T(t)Y_0, x^*\ra + \la 
\int_0^{t}T(t-u)F(Y(u))du, x^*\ra\\&+ \int_{0}^{t}G^*(Y(u))T^\odot(t-u)x^*dW_H(u). \label{r5} 
\end{align}
Using the hypothesis \ref{h0}, \ref{F} and \ref{G} by the Krein-Smulyan theorem the above equality is also valid for all $x^*\in E^*$ (see \cite[Corollary 6.6]{Kunze} and the proof of Lemma 2.7 in \cite{vanNeervenVeraarWeis}).
 
Now we assume that hypothesis (i) holds and we prove that for all $t>0$ the process
 $u\mapsto \Psi(u):=T(t-u)G(Y(u))$ is stochastically integrable on $[0,t]$ with respect to $W_H$.  
From the Step 1  of the proof it follows that the process $\Psi$ is scalarly in $L^q(\Omega;L^2(0,t;H))$. We define the random variable $\xi$ by $\xi:=Y(t)-T(t)Y_0-\int_0^{t}T(t-u)F(Y(u))du$ and claim that $\xi\in L^q(\Omega;E)$. Indeed,  by the assumptions \ref{h0}, \ref{F} and (i) we obtain 
\begin{align*}
\left(\E \n \xi \n_E^q\right)^\inv{q} &\leq \left(\E \n Y(t)\n^q_E\right)^\inv{q}+\n T(t)\n_{\Ca{L}(E)}\left(\E \n Y_0\n^q_E\right)^\inv{q}\\&\quad\quad+\left(\E\left( \int_0^{t}a(t-u)(1+\n Y(u)\n^2_E)du\right)^q\right)^\inv{q}
\\&\leq \left(\E \n Y(t)\n^q_E\right)^\inv{q}+\n T(t)\n_{\Ca{L}(E)}\left(\E \n Y_0\n^q_E\right)^\inv{q}\\&\quad+\n a\n_{L^1(0,t)}\sup_{u\in[0,t]}\left(\E \ (1+\n Y(u)\n^2_E)^q\right)^\inv{q}<\infty, 
\end{align*}
where in the last inequality we use the Minkowski integral inequality. From \eqref{r5} it follows that $\xi$ satisfies \eqref{xi} for all $x^*\in E^*$.
Hence,  by Theorem 3.6 and Remark 3.8 in \cite{vanNeervenVeraarWeis},
 the process $\Psi$ is stochastically integrable on $[0,t]$ with respect to $W_H$ and condition (ii) holds. 

Notice that by Lemma \ref{l:*weakdense} the set $D((A^\odot)^2)$ separates the points of $\tilde{E}$, thus also in $E$,  and $E$ is assumed to be separable, hence by the Hahn-Banach theorem there exists a sequence $(x^*_n)_{n\geq 1}$ of elements from $D((A^\odot)^2)$ which separates the points of $E$. Thus \eqref{r5} holds simultaneously for all $x_n^*$ on set of measure one. Therefore \eqref{voc2} holds.\par 
On the other hand assume that $Y$ is a  mild solution to \eqref{SCP}. By Remark \ref{r:varcons}.(ii) it follows that for all $x^*\in D(A^*)$ the process  $u\mapsto G^*(Y(u))x^*$ is stochastically integrable. Moreover, by  \eqref{voc2} and then by Fubini's theorem and \eqref{r1}, and once more by \eqref{voc2} we obtain, almost surely,
\begin{align*}
&\la\int_0^tY(s)ds,  A^\odot x^*\ra
\stackrel{\eqref{voc2}}{=}\la \int_0^tT(s)Y_0ds, A^\odot x^*\ra\\
&\hspace{2cm}\quad+\la \int_0^t\int_0^s T(s-u)F(Y(u))duds, A^\odot x^*\ra\\
&\hspace{2cm}\quad+ \int_0^t\int_0^s G^*(Y(u))T^\odot(s-u)A^\odot x^*dW_H(u)ds \\ 
&\stackrel{\eqref{r1}}{=}\la T(s)Y_0, x^*\ra-\la Y_0, x^*\ra+\la \int_0^t(T(t-u)F(Y(u))du, x^*\ra\\&\hspace{1.5cm}-\int_0^t \la F(Y(u)), x^*\ra du \\
&\hspace{1.5cm}+\la \int_0^tT(t-u)G(Y(u))dW_H(u),x^*\ra-\int_0^tG^*(Y(u))x^*dW_H(u)\\
&\stackrel{\eqref{voc2}}{=}\la Y(t)-Y(0),x^*\ra-\int_0^t \la F(Y(u)),x^*\ra du-\int_0^tG^*(Y(u))x^*dW_H(u). 
\end{align*}
\end{proof}
In  a separable Banach space $E$ let us consider the version of \eqref{SCP}, where the noise is introduced additively i.e. $G\in\Ca{L}(H,\tilde{E})$. We will denote it by (SCPa).
Here we do not need the assumption that $E$ has \textsc{umd}$^-$ property, since stochastic Wiener integral in every Banach space is characterised by $\gamma$-norms (see \cite[Theorem 4.2]{vanNeervenVeraarWeis2}).  By Theorem \ref{t:varcons} we obtain the following corollary. 
\begin{cor}\label{c:equivSol}
Assume that condition \ref{F} is satisfied,  $T(s)G\in\Ca{L}(H,E)$ for all $s>0$ and $Y$ is an $E$-valued $H$-strongly measurable adapted process with locally Bochner square integrable trajectories a.s. If for some $t_0> 0$ one of the following conditions holds 
\begin{enumerate}
\item  $\sup_{s\in[0,t_0]} \E \n Y(s)\n^q_E<\infty$ for some $q>1$;

\item the function $u\mapsto T(u)G$ represents an operator in $\gamma(L^2(0,t_0;H),E)$,
\end{enumerate} 
then, 
$Y$ is a weak solution to (SCPa) if and only if $Y$ is a mild solution to (SCPa).\\ 
Moreover, if there exists a solution of (SCPa) and the hypothesis (i) holds, then $u\mapsto T(t-u)G$ represents an element in $\gamma(L^2(0,t;H),E)$ for all $t>0$.
\end{cor}
\begin{rem}\label{r:gammaTG}
Notice that by Theorem 7.1 in~\cite{vanNeervenVeraarWeis2} it follows that if there exists $t_0> 0$ such that 
$u\mapsto T(u)G$ represents an element in $\gamma(L^2(0,t_0;H),E)$, then  for all $t> 0$ the function $[0,t]\ni u\mapsto T(u)G$ also belongs to $\gamma(L^2(0,t;H),E)$ .
\end{rem}

\section{Equivalence of generalised strong, weak and mild solutions}
\label{s:threesol}
In \cite{coxgorajski} a generalised strong solution to \eqref{SCP} is defined and its equivalence to a mild solution of \eqref{SCP} is proven. Under weaker assumptions we establish in Theorem \ref{t:varcons2} the equivalence of mild, weak and generalised strong solutions. First extend hypothesis \ref{F}.
\renewcommand\theenumi {\Alph{enumi}}
\renewcommand\theenumii {\roman{enumii}}
\renewcommand\labelenumi{(H\theenumi') }
\labelformat{enumi}{(H#1')}
\begin{enumerate}
    \item \label{F'} Assume that $F:D(F)\subset E\to E$ satisfies \ref{F}  and for all $t>0$ and $g\in L^1(0,t;E)$ the function  $F(g)$ is Bochner integrable on $[0,t]$. 
\end{enumerate}
It is clear that if $F$ is a Lipschitz function, then \ref{F'} is satisfied.
\renewcommand\theenumi {\roman{enumi}}
\renewcommand\labelenumi{(\theenumi)}

\begin{thm}\label{t:varcons2}
Assume that the hypotheses of Theorem \ref{t:varcons} are satisfied and condition \ref{F'} holds. Let $Y$ be an $E$-valued $H$-strongly measurable adapted process with locally Bochner square integrable trajectories a.s. If for all $t> 0$ the processes:
\begin{align}\label{G2}
u\mapsto G(Y(u)),\quad u\mapsto \int_0^{t-u}T(s)G(Y(u,\omega))ds
\end{align}
 are in $\gamma(L^2(0,t;H),E)$ a.s.,
then the following conditions are equivalent:
\begin{enumerate}
	\item $Y$ is a generalised strong solution of \eqref{SCP}.
    \item $Y$ is a weak solution of \eqref{SCP}.
	 	 \item $Y$ is a mild solution of \eqref{SCP}.
\end{enumerate}  
\end{thm}
In the case where $E$ is reflexive  Theorem \ref{t:varcons2} is a simple consequence of Theorem  \ref{t:varcons} (see Remark \ref{u:repr2} below).   
\begin{rem}
\label{u:repr2}
\begin{enumerate}
    \item It is obvious that if $Y$ is a generalised strong solution to \eqref{SCP}, then $Y$ is a weak solution of \eqref{SCP}.
    
\item Let $(T(t))_{t\geq 0}$ be $\gamma$-bounded. In Hilbert spaces uniformly bounded families are $\gamma$-bounded (for the definition of 
$\gamma$-boundness and more on the applications of this notion see \cite{Denk}). For all $t>0$ the family $$\Big\{\int_0^t 1_{[0,t-u]}(s)T(s)ds: u\in[0,t]\Big\}$$ is also $\gamma$-bounded as an integral mean of $\gamma$-bounded operators (see Theorem 9.7 in \cite{vanNeervenISEM}). Hence using the multiplier theorem due to Kalton and Weis \cite{Kalton} we obtain: if $G(Y(\cdot))$ is in $\gamma(L^2(0,t;H),E)$ a.s., then the processes $$T(t-\cdot)G(Y(\cdot)),\quad \int_0^{t-\cdot}T(s)G(Y(\cdot,\omega))ds$$ are in $\gamma(L^2(0,t;H),E)$ a.s. 

\item Let $\tilde{E}$ be a reflexive Banach space, and for all $t>0$ the process $ G(Y(\cdot))$ is in $\gamma(L^2(0,t;H),E)$ a.s. Assume that  and  \ref{F'}, \ref{G} are satisfied and $Y$ has almost all trajectories locally square integrable. Then, it is easy to prove that a weak solution to \eqref{SCP} is a generalised strong solution to \eqref{SCP}. Indeed, let $Y$ be a weak solution to \eqref{SCP}, then for every $t>0$ and all $x^*\in D(A^*)$ we have the equality
\begin{align}\label{Erefl}
\la\int_0^tY(s)ds, A^*x^* \ra =\la e(t), x^* \ra \quad \textrm{a.s.,}
\end{align}
where $e(t)= Y(s)-Y_0-\int_0^tF(Y(s))ds-\int_0^tG(Y(s))ds\in E\subset \tilde{E}= \tilde{E}^{**}$. By reflexivity of $\tilde{E}$ it follows that  $D(A^*)$ is dense in $\tilde{E}^*$, hence, almost surely, the right hand side of \eqref{Erefl} has an extension to bounded linear functional on $\tilde{E}^*$.  Thus by the definition of $A^*$ one has $\int_0^tY(s)ds\in D(A^{**})$ and $A^{**}\int_0^tY(s)ds=e(t)$ a.s. Finally, by reflexivity of $\tilde{E}$ we can replace in the last equality  $(A^{**},D(A^
{**}))$ by $(A,D(A))$ and the assertion follows (see B.10 in \cite{EngNag}).    
\end{enumerate}
\end{rem}

\begin{proof}[Proof of Theorem \ref{t:varcons2}]
By Theorem \ref{t:varcons} and Remark \ref{u:repr2}.(i)  it suffices to prove that every mild solution  to \eqref{SCP} is a generalised strong solution of \eqref{SCP}. 

Fix $t> 0$. Let $Y$ be a mild solution of \eqref{SCP} satisfying the assumptions of Theorem \ref{t:varcons2}. Observe that $[0,t]\times \Omega\ni (u,\omega)\mapsto\int_0^t\Psi_1(s,u,\omega)ds$, where $\Psi_1$ is a process defined in Step 1 of the proof of Theorem \ref{t:varcons}, satisfies the assumptions of  Lemma 2.8 in \cite{coxgorajski}, i.e. for all $h\in H$ process $\Phi_1(s)h\in D(A)$ a.s. and the processes
\begin{align*}
 u\mapsto\int_0^t\Psi_1(s,u,\omega)ds=&\int^{t-u}_0 T(s)G(Y(u))ds\in D(A), \\
u\mapsto A\int_0^t\Psi_1(s,u,\omega)ds=&T(t-u)G(Y(u))-G(Y(u))\textrm{ a.s.}
\end{align*}
represent elements in $\gamma(L^2(0,t;H),E)$ a.s. Hence from Lemma 2.8 in \cite{coxgorajski} it follows that
\begin{align}
\label{r:Asint} \int_0^t\int^{t-u}_0 T(s)G(Y(u))dsdW_H(u)&\in D(A),\\
A\int_0^t\int^{t-u}_0T(s)G(Y(u))dsdW_H(u)&=\int_0^t T(t-u)G(Y(u))dW_H(u)\notag\\&-\int_0^tG(Y(u))dW_H(u) \textrm{ a.s.} \notag 
\end{align}
Moreover, by \ref{F'} and the properties of strongly continuous semigroup we obtain, almost surely,
\begin{align}
\notag
&\int_0^t\int^{t-u}_0 T(s)F(Y(u))dsdu, \int_0^tT(s)Y_0ds\in D(A),\\
&A\int_0^t\int^{t-u}_0 T(s)F(Y(u))dsdu=\int_0^tT(t-u)F(Y(u))du-\int_0^tF(Y(u))du, \notag \\
&A\int_0^tT(s)Y_0ds=T(t)Y_0-Y_0. \label{r:Aint}
\end{align}
Therefore, by \eqref{voc2} and \eqref{r:Asint}-\eqref{r:Aint} we have
\begin{align*}
&\int_0^tY(s)ds\in D(A),\\
&A\int_0^tY(s)ds\stackrel{\eqref{voc2},\eqref{r:Asint}-\eqref{r:Aint}}{=}T(t)Y_0-Y_0+\int_0^tT(t-u)F(Y(u))du\\
&\quad-\int_0^tF(Y(u))du+\int_0^t\left[T(t-u)G(Y(u))-G(Y(u))\right]dW_H(u)\\
&\stackrel{\eqref{voc2}}{=} Y(t)-Y_0 -\int_0^tF(Y(u))du-\int_0^tG(Y(u))du \textrm{ a.s.}
\end{align*}
\end{proof}

\begin{cor}\label{c:equivSol2} Under the assumptions of Corollary \ref{c:equivSol} if, in addition, \ref{F'} holds and for some $t_0> 0$ the mapping:
\begin{align}\label{G1aa}
 u\mapsto \int_0^{t_0-u}T(s)Gds
\end{align}
 represents an element in $\gamma(L^2(0,t_0;H),E)$, then the notions of generalised, weak and mild solutions to (SCPa) are equivalent.
\end{cor}

\section{Existence, uniqueness and continuity of solutions to stochastic delay equations}\label{s:SDEaddtive}

In this section we apply the results from Sections 2, 3  to establish the existence of a  unique continuous solution to a stochastic delay evolution equation of the form:
\begin{equation}\label{SDEa}
\eqtwo{dX(t)=BX(t)dt+\phi(X(t),X_t)dt+\psi dW_H(t),}{t>0,}{X(0)=x,X_0=f,}{}    
\end{equation}
where $(B,D(B))$ generates a semigroup of linear operators $(S(t))_{t\geq 0}$ on a separable Banach space $\tilde{E}$, $X_t:\Omega\times[-1,0]\to E$ is a segment process defined as $X_t(\theta)=X(t+\theta)$, $\theta \in [-1,0]$, and $E$ is a separable Banach space such that the hypothesis \ref{h0} holds with $A:=B$. We will use the following assumptions on mappings $\phi$  and $\psi$:

\renewcommand\theenumi {\arabic{enumi}}
\renewcommand\theenumii {\alph{enumii}}
\renewcommand\labelenumi{(H$\phi$) }
\labelformat{enumi}{(H$\phi$)}
\begin{enumerate} 
    \item \label{hphi} $\phi:D(\phi)\subset \Ep{p}\to \tilde{E}$, where $\Ep{p}=E\times L^p(-1,0;E)$ and $p\geq 1$,  is densely defined mapping and there exists $a\in L^{p}_{loc}(0,\infty)$ such that for all $t>0$ and $\mathcal{X}, \mathcal{Y} \in D(\phi)$ we have $S(t)\phi(\mathcal{X})\in E$   
\begin{align*}
&\norm{S(t)\phi(\mathcal{X})}_E\leq a(t)(1+\|\mathcal{X}\|_{\Ep{p}}), \\
 &\norm{S(t)(\phi(\mathcal{X})-\phi(\mathcal{Y})}_E\leq a(t)\|\mathcal{X}-\mathcal{Y}\|_{\Ep{p}},
\end{align*} 
\end{enumerate}

\renewcommand\theenumi {\arabic{enumi}}
\renewcommand\theenumii {\alph{enumii}}
\renewcommand\labelenumi{(H$\psi$)}
\labelformat{enumi}{(H$\psi$)}

\begin{enumerate} 
    \item \label{hpsi}  $\psi\in\Ca{L}(H,\tilde{E})$ and the mapping  $u\mapsto S(u)\psi\in\Ca{L}(H,E)$ represents an element in $\gamma(L^2(0,t;H),E)$ for some $t>0$. 
\end{enumerate}
\renewcommand\theenumi {\roman{enumi}}
\renewcommand\theenumii {\alph{enumii}}
\renewcommand\labelenumi{(\theenumi) }

\begin{defin}
A strongly measurable adapted\footnote{For all $t\in[-1,0]$ we assume that $X(t)$  is a $\Ca{F}_0$-strongly measurable.} process $X:[-1,\infty)\times\Omega \rightarrow E$  is called a \emph{weak solution} to \eqref{SDEa} if $X$ belongs to $L_{loc}^p(0,\infty;E)$ a.s. and for
all $t> 0$ and $x^*\in D(B^\odot)$:
 \begin{enumerate}
\item $s\mapsto\langle \phi(X(s),X_s), x^*\rangle$ is integrable on $[0,t]$ a.s.; 
 \item $(s,\omega)\mapsto\psi^*(X(s),X_s)x^*$ is stochastically integrable on $[0,t]$;
\item  almost surely 
\begin{align}\notag 
\langle X(t), x^*\rangle -\langle x_0,x^* \rangle&=\int_0^t\langle X(s), B^\odot x^*\rangle ds \\  
&+\int_0^t\langle \phi(X(s),X_s), x^*\rangle ds
+W_H(t)\psi^*x^*
\label{weaksolSDE}
\end{align}
\item $X_0=f_0$.
\end{enumerate}
\end{defin}

In \cite{coxgorajski} and \cite{GorajskiSDE} the Markovian representation of stochastic delay evolution equations with state dependent noise (i.e. $\psi:=\psi(X(t),X_t)$) in type 2 \textsc{umd} Banach spaces is proven. Using the same arguments we obtain the following representation for \eqref{SDEa}.
\begin{thm}[\cite{GorajskiSDE}]
 \label{t:rep}
Let $p>1$. The following conditions hold.
\begin{enumerate}
\item If $X$ is a weak solution to \eqref{SDEa}, then the process $Y$ defined by $Y(t):=[\pi_1Y(t),\pi_2Y(t)]'=[X(t),X_t]'$ is a weak solution to a stochastic evolution equation in $\Ep{p}=E\times L^p(-1,0;E)$: 
\begin{equation}\label{SCPDa} 
\eqtwo{dY(t)=(AY+F(Y))dt+GdW_H(t),}{t>0,}{Y(0)=[x,f]',}{}    
\end{equation}
where prime is a transposition, $G=[\psi,0]'\in\mathcal{L}(H,\tilde{\Ep{p}})$,  $F:\Ep{p}\to \tilde{\Ep{p}}$, $F=[\phi,0]'$ and $\lee A=\vecfour{B}{\ 0}{0}{\frac{d}{d\theta}}, D(A) \p$ is the generator of the delay semigroup $T(t)=\vecfour{S(t)}{\ 0}{\Ca{S}_t}{T_l(t)}$ on $\tilde{\Ep{p}}=\tilde{E}\times L^p(-1,0;\tilde{E})$, where $\lee T_l(t)\p _{t\geq 0}$ is the left translation semigroup on $L^p(-1,0;\tilde{E})$ and $\Ca{S}_s\in\Ca{L}(\tilde{E},L^p(-1,0;\tilde{E}))$ is given by
 \begin{align}\label{S_t}
 (\Ca{S}_sx)(\theta)=\eqtwo{0}{\theta\in(-1,-s\maxsym -1)}{S(\theta+s)x}{\theta\in(-s\maxsym-1,0)}
\end{align}
for all $s\geq0$ and all $x\in \tilde{E}$ (cf. Theorem 3.25 in \cite{Batkai2005}).\label{repr1}
\item Assume that the hypotheses \ref{hphi} and \ref{h0} with $A:=B$ hold and   $S(s)\psi\in\Ca{L}(H,E)$ for all $s>0$. Let $Y$ be an $\Ep{p}$-valued $H$-strongly measurable adapted process with locally Bochner square integrable trajectories a.s. such that for some $t_0>0$ one of the following conditions holds
\begin{enumerate}
\item  $\sup_{s\in[0,t_0]} \E \n Y(s)\n^q_{\Ep{p}}<\infty$ for some  $q>1$ . 

\item the function $u\mapsto T(u)G$ belongs to $\gamma(L^2(0,t_0;H),\Ep{p})$.
\end{enumerate}
Then, if $Y$
 is a weak solution to \eqref{SCPDa}, then the process defined by $X|_{[-1,0)}=f_0$, $X(t):=\pi_1Y(t)$ for $t\geq 0$ is a weak solution to \eqref{SDEa}. \label{repr2}
\end{enumerate}

\end{thm}
 
\begin{proof}
The part (i) may be proved in much the same way as the corresponding part of Theorem 3.9 in \cite{GorajskiSDE}.

For the proof of the second part it is enough to  show that weak and mild solutions to \eqref{SCPDa} are equivalent (see the proof of Theorem 3.9 in \cite{GorajskiSDE}). Using the Lemma 3.1 in \cite{GorajskiSDE} one can prove that  the condition \ref{F} holds, hence we can apply the Corollary \ref{c:equivSol} to obtain the desired equivalence of solutions.
\end{proof}   

The  delay equations play a crucial role in modelling phenomena e.g. in bioscience (cf.\cite{Baker}, \cite{Erneux}) economics and finance (\cite{Mao}, \cite{Mohammed}).
 Here we consider delay evolution equation with the Wiener additive noise, for delay equation in \textsc{umd} type 2 Banach space with more general Wiener noises we refer to \cite{coxgorajski},\cite{GorajskiSDE}, where the reader can also find a more extensive literature overview. For stochastic delay evolution equation with infinite delay see \cite{Crewe}.
 At the end of this section we give two examples of stochastic delay partial differential equation in non-reflexive Banach space. First, in $C_0([0,1])$ we examine a simple stochastic delay advection-reaction equation. This equation can be used to model product goodwill (see \cite{BarucciGozzi1999} for deterministic goodwill model without delays). In the second example in the state space $L^1(0,\infty)$ we analyse a stochastic delay age-dependent equation of the Sharpe--Lotka-–McKendrick (or von Foerster) type (see \cite{Webb1985} and \cite{Qi-Min}).    
 
 In Section \ref{s:4.1} we recall the existence and continuity results for stochastic evolution equations with additive noise in separable Banach spaces. 
 
\subsection{Stochastic evolution equation with additive noise} \label{s:4.1}
Notice that Corollary \ref{c:equivSol} yields the following result concerning the existence and uniqueness of weak solution to (SCPa) i.e. \eqref{SCP} where $G\in\Ca{L}(H,\tilde{E})$ (see   \cite{DaPratoZabczyk}, \cite{PeszZab} and \cite{BrzezniakNeerven_stochconv}, \cite{vanNeervenVeraarWeis2} for the linear case.)
\begin{thm}\label{t:existence}
Let $q\geq 1$. Assume that conditions \ref{F} and \ref{h0} are satisfied, $T(s)G\in\Ca{L}(H,E)$ for all $s>0$ and $Y$ is an $E$-valued $H$-strongly measurable adapted process with locally Bochner square integrable trajectories a.s.
The following assertions are equivalent:
\renewcommand\theenumi {\roman{enumi}}
\renewcommand\labelenumi{(\theenumi)}
\begin{enumerate}
    \item[(1)] For every $t>0$ and all $y\in L^q((\Omega,\F_0);E)$ there exists a unique weak solution   $Y(\cdot;y)$  to (SCPa) in $\Bo{SL}_{\Ca{F}}^q(0,t;E)$;
    \item[(2)] The function $u\mapsto T(u)G$ represents an element in $\gamma(L^2(0,t;H),E)$ for some $t>0$.
\end{enumerate}
Moreover, if the solution $Y$ exists then there exists  $L>0$ such that for all
$x,y\in L^{q}(\Omega;E)$ and  $s\geq 0$:
\begin{align*}
&\sup_{s\in[0,t]}\E \ \norm{Y(s;x)}^{q}\leq L(1+ \E\ \norm{x}^{q}),\\
&\sup_{s\in[0,t]}\E \ \norm{Y(s;x)-Y(s;y)}^{q}\leq L\ \E\ \norm{x-y}^{q};
\end{align*} 
 the probability distribution of $Y(s;x)$ does not depend on cylindrical Wiener process $W_H$ and the underlying probability space.
\end{thm}
\begin{proof}
The implication (1)$\Rightarrow$(2) follows from the second part of Corollary \ref{c:equivSol}.

For the proof of implication (2)$\Rightarrow$(1) let us fix  $q\geq 1$, $t>0$ and $y\in L^q((\Omega,\F_0);E)$ and assume (2). By Remark \ref{r:gammaTG} the condition (2) holds for all $t>0$. We define a mapping $\Ca{K}$ by 
$$\mathcal{K}(Z)(s)=T(s)y+\int_0^sT(s-u)F(Z(u))du+\int_0^s T(s-u)GdW_H(u)$$ for all $Z\in \Bo{SL}_{\Ca{F}}^{q}(0,t;E)$. The symbol $\Bo{SL}_{\Ca{F}}^{q}(0,t;E)$ stands for the Banach space of strongly measurable, adapted process $Y$ with the Bilecki's type norm 
$\n Y \n_{\beta}=\sup_{s\in[0,t]} e^{-\beta s}\lee \E\ \n Y(s)\n_{\Ep{p}}^{q}\p^{\inv{q}}$ for some $\beta>0$. 
By assumptions, it follows that both stochastic and Bochner integrals in the definition of $\Ca{K}$ are well defined.  The first term of $\Ca{K}$ is continuous a.s. Corollary 6.5 in~\cite{vanNeervenVeraarWeis2} yields that the stochastic convolution in  $\Ca{K}$ is a continuous process in $q$-th moment. Moreover, by \ref{F} and the Minkowski's integral inequality for all $Z\in \Bo{SL}_{\Ca{F}}^{q}(0,t;E)$ and for every $s\in[0,t]$ one gets   
\begin{align*}
e^{-s\beta}&\lee\E\ \norm{\int_0^s T(s-u)F(Z(u))du}_E^q\p^\inv{q}\leq \notag\\
&\leq e^{-s\beta}\int_0^s a(s-u)e^{\beta u} e^{-\beta u}\lee\E\ \lee 1+\norm{ Z(u)}_{E}\p^q\p^\inv{q}du \notag\\
 \notag
&\leq\lee 1+\n Z \n_{\beta}\p \int_0^s a(u)e^{-\beta u}du.
\end{align*}
Hence
\begin{align}
\label{ist:bochner1}
\sup_{s\in[0,t]}e^{-s\beta}\lee\E\ \norm{\int_0^s T(s-u)F(Z(u))du}^q\p^\inv{q}\leq C_{\beta,a}\lee 1+\norm{ Z }_{\beta}\p, 
\end{align}
where $C_{\beta,a}=\int_0^t \tilde{a}(u)e^{-\beta u}du$.
Between the same lines using \ref{F} for all $Z_1,Z_2\in \Bo{SL}_{\Ca{F}}^{q}(0,t;E)$ one has
\begin{align}\label{ist:bochner2}
\sup_{s\in[0,t]}\norm{\int_0^s T(s-u)\lee F(Z_1(u))-F(Z_1(u))\p du}_\beta\leq C_{\beta,a}\norm{ Z_1 -Z_2}_{\beta}. 
\end{align}
Hence for $\beta>0$ large enough the operator $\Ca{K}$ is a strict contraction in  $\Bo{SL}_{\Ca{F}}^{q}(0,t;E)$. Therefore, the existence and uniqueness results follows by  the Banach fixed-point theorem and by Corollary \ref{c:equivSol}.   
Analysis similar to that in the proof of Theorem 9.29 in \cite{PeszZab} shows 
that the second part of theorem is true.
\end{proof}

Using the factorization
method as introduced in Section 2 of \cite{DaKwaZab87} and Theorem 3.4 in \cite{coxgorajski} (see also Theorem 3.3 in \cite{VeraarZimmerschied}) we obtain sufficient condition for continuity of a solution to (SCPa).
\begin{thm}\label{t:reg_aSPC}
Let $q>2$. Under the hypotheses of Theorem \ref{t:existence}, if, in addition,  there exists $\alpha\in(\inv{q},\inv{2})$ such that for all $t>0$ the function $a$ (see assumption \ref{F}) satisfies
\begin{align*}
&\int_0^ta(s)s^{-\alpha}ds<\infty, \\
&\sup_{s\in[0,t]}\n u\mapsto (s-u)^{-\alpha}T(s-u)G\n_{\gamma(L^2(0,s;H),E)}<\infty,
\end{align*}
then for all $y\in L_{\Ca{F}_0}^q(\Omega;E)$  the weak solution $Y=Y(;y)$ of (SCPa) belongs to  $L^q(\Omega;C([0,t];E))$.
Moreover, there exists  $L>0$ such that 
for all $x,y\in L^{q}(\Omega;E)$
\begin{align}\label{r:ineqPocz1}
&\E \ \sup_{s\in[0,t]}\norm{Y(s;x)}_{E}^{q}\leq L \big(1+ \E\ \norm{x}_{L^{q}(\Omega;E)}^{q}\big),\\
&\E \ \sup_{s\in[0,t]}\norm{Y(s;x)-Y(s;y)}_{E}^{q}\leq L\ \E\ \norm{x-y}_{L^{q}(\Omega;E)}^{q}. \label{r:ineqPocz2}
\end{align}
\end{thm}

\subsection{Stochastic delay evolution equation}
\renewcommand\theenumi {\roman{enumi}}
\renewcommand\theenumii {\alph{enumii}}
\renewcommand\labelenumi{(\theenumi) }

 Using Theorems \ref{t:existence} and \ref{t:reg_aSPC} (see also Corollaries 3.12 and 3.13 in \cite{GorajskiSDE}) we obtain the proposition. 
\begin{prop}\label{p:}
Let $q\geq 1$ and $p\geq 1$. Assume that conditions \ref{hphi} and \ref{h0} with $A:=B$ are satisfied, $S(s)\psi\in\Ca{L}(H,E)$ for all $s>0$ and $X$ is an $E$-valued $H$-strongly measurable adapted process with locally Bochner $p\maxsym 2$-power integrable trajectories a.s.\\ \noindent
Then, the hypothesis \ref{hpsi} holds if and only if for all $t>0$ and all $x\in L_{\Ca{F}_0}^{p\maxsym q}(\Omega;E)$ and $f\in L_{\Ca{F}_0}^{q}(\Omega;L^p(-1,0;E))$   there exists a unique weak solution   $X(\cdot;x,f)$  to \eqref{SDEa} in the Banach space $\Bo{SL}_{\Ca{F}}^{p\maxsym q}(0,t;E)$.\\ \noindent
Moreover, if the solution $X(\cdot;x,f)$ exists, then it satisfies, almost surely, 
\begin{align}\label{SDEsol}
X(t) = S(t)x+\int_{0}^{t} S(t-s)\phi(X(s),X_s)ds+ \int_{0}^{t} S(t-s)\psi dW_H(s),  
\end{align}
the dependence of $X(;x,f)$ on initial conditions as in Theorem \ref{t:existence} holds and for all $s\geq 0$ the probability distribution of $X(s;x,f)$ does not depend on cylindrical Wiener process $W_H$ and the underlying probability space.

Furthermore, if $(q\maxsym p)>2$ and there exists $\alpha\in(\inv{q},\inv{2})$ such that
\begin{align}\label{hpsi_alpha_1}
&\int_0^ta(s)s^{-\alpha}ds<\infty, \\
\textrm{where the function }  a &\textrm{ is defined  in assumption \ref{hphi} and} \notag\\\label{hpsi_alpha_2}
\sup_{s\in[0,t]}\n &u\mapsto (s-u)^{-\alpha}S(s-u)\psi\n_{\gamma(L^2(0,s;H),E)}<\infty;
\end{align}
then the weak solution $X(\cdot;x,f)$ to \eqref{SDEa} belongs to $L^{p\maxsym q}(\Omega;C([0,t];E))$ and the inequities of type \eqref{r:ineqPocz1}-\eqref{r:ineqPocz2} hold. 
\end{prop}
\begin{proof}
In the proof we use Theorem \ref{t:rep} and then we apply Theorems \ref{t:existence} and \ref{t:reg_aSPC} to problem \eqref{SCPDa}. Let $t>0$. We shall now check the assumptions of these theorems. From the $p\maxsym 2$-power integrability of the process $X$ it follows, by 
Remark 4.7 in \cite{coxgorajski}, that a weak solution $Y=[X,X_t]'$ to \eqref{SCPDa} has square integrable trajectories a.s. Moreover, by Minkowski's integral inequality if $X\in \Bo{SL}_{\Ca{F}}^{p\maxsym q}(0,t;E)$, then $Y=[X,X_t]'\in \Bo{SL}_{\Ca{F}}^{p\maxsym q}(0,t;E)$ (see the proof of Corollary 3.12 in \cite{GorajskiSDE}). Finally, notice that from \ref{hphi} it follows that condition \ref{F} is satisfied for  $F=[\phi,0]'$. 

What is left to show is that  $[0,t]\ni u\mapsto T(u)[\psi,0]'$ represents an operator $[R_{\pi_1},R_{\pi_2}]'$ in  $\gamma(L^2(0,t;H),\Ep{p})$ if and only if condition  \ref{hpsi} holds. Indeed, by the properties of delay semigroup $(T(t))_{t\geq 0}$ (see (9) in \cite{GorajskiSDE} and Proposition 3.11 in \cite{Batkai2005}) we have $\pi_1T(u)[\psi,0]^{'}=S(u)\psi$ and $(\pi_2T(u)[\psi,0]^{'})(\theta)=1_{(u+\theta>0)}S(u+\theta)\psi$ for every $u>0$ and a.e. $\theta\in[-1,0]$. Hence by \ref{hpsi} it follows that $R_{\pi_1} \in \gamma(L^2(0,t;H),E)$.  
Furthermore, for all $f\in L^2(0,t;H)$ using Lemma 3.4 in \cite{GorajskiSDE} we obtain
\begin{align}\label{r:Rp2}
 \lee R_{\pi_2} f\p(\theta)=&\lee\int_0^t \pi_2\Ca{T}(u)[\psi f(u),0]^{'} du\p(\theta)=\int_{-\theta}^tS(u+\theta)\psi f(u)du\\
=&\int_0^tS(u)\psi P_{\theta}f(u)du=R_{\pi_1}P_\theta f, \quad \text{a.e. }\theta \in[-1,0], \notag
\end{align}
where $P_{\theta}\in\Ca{L}(L^{2}(0,t;H))$ for all $\theta\in[-1,0]$ is defined by $$(P_{\theta}f)(u)=1_{(0,t+\theta)}(u)f(u-\theta)\quad  \text{a.e. } u\in[0,t].$$
By  $\gamma$-Fubini isomorphism (see Proposition 2.6 in \cite{vanNeervenVeraarWeis})  $$R_{\pi_2}\in\gamma(L^2(0,t;H),L^p(-1,0;E))$$ if and only if
$$\int_{-1}^0\norm{ \lee R_{\pi_2}\cdot\p(\theta)}^p_{\gamma(L^2(0,t;H),E)}d\theta<\infty.$$   
Since  $\n P_\theta\n_{\Ca{L}(L^{2}(0,t;H))}\leq 1$ and using the ideal property of $\gamma$-radonifying operators, form \eqref{r:Rp2} we get
\begin{align*}
\int_{-1}^0\norm{ \lee R_{\pi_2}\cdot\p(\theta)}^p_{\gamma(L^2(0,t;H),E)}d\theta&\leq\int_{-1}^0\norm{ R_{\pi_1}}^p_{\gamma(L^2(0,t;H),E)}\norm{ P_\theta}^p_{\Ca{L}(L^{2}(0,t;H))}d\theta\\&\leq
\norm{R_{\pi_1}}^p_{\gamma(L^2(0,t;H),E)}<\infty.
\end{align*}
\par
To prove the last assertion of theorem we show that for all $\alpha>0$  the equivalence 
\begin{align}
&\sup_{s\in[0,t]}\norm{ u\mapsto (s-u)^{-\alpha}S(s-u)\psi}_{\gamma(L^2(0,s;H),E)}<\infty \label{alpha1}\\ &\hspace{3cm}\Updownarrow \notag \\ &\sup_{s\in[0,t]}\norm{ u\mapsto (s-u)^{-\alpha}T(s-u)G}_{\gamma(L^2(0,s;H),\Ep{p})}<\infty\label{alpha2}
\end{align} 
holds.
It is clear that the implication \eqref{alpha2} $\Rightarrow$ \eqref{alpha1} is true. 
For fixed  $s\leq t$ let us denote by $[R_{s,\alpha,\pi_1},R_{s,\alpha,\pi_2}]^{'}$ the operator in $\gamma(L^2(0,s;H),\Ep{p})$ which is represented by  $$[0,s]\ni u \mapsto (s-u)^{-\alpha} T(s-u)G=\left[(s-u)^{-\alpha} S(s-u)\psi,(s-u)^{-\alpha} \Ca{S}_{s-u}\psi\right]',$$
where $\Ca{S}_{s}$ is defined by \eqref{S_t}.  Then, in much the same way as in  \eqref{r:Rp2} for all $f\in L^2(0,s;H)$ and a.e. $\theta \in[-1,0]$ we get
\begin{align*}
 \lee R_{s,\alpha,\pi_2} f\p(\theta)=&\lee\int_0^s (s-u)^{-\alpha}\pi_2 T(s-u)Gf(u) du\p(\theta)\\&=\int_{0}^{s+\theta}(s-u)^{-\alpha}S(s-u+\theta)\psi f(u)du, 
\end{align*}
Since $(s-u)^{-\alpha}\leq (s-u+\theta)^{-\alpha}$ for all $u\leq s+\theta$, by the ideal property of $\gamma$-radonifying operators for a.e $\theta \in[-1,0]$ we obtain
$$\norm{\lee R_{s,\alpha,\pi_2} \cdot\p(\theta)}_{\gamma(L^2(0,s;H),E)}\leq \norm{(R_{s+\theta,\alpha,\pi_1} \cdot)(\theta)}_{\gamma(L^2(0,s+\theta;H),E)}. $$
Hence
\begin{align*}
\int_{-1}^0\n (R_{s,\alpha,\pi_2}\cdot)(\theta)\n^p_{\gamma(L^2(0,s;H),E)}d\theta&\leq\int_{-1}^0\n R_{s+\theta,\alpha,\pi_1}\cdot\n^p_{\gamma(L^2(0,s+\theta);H,E)}d\theta\\&\leq
\sup_{s\in[0,t]}\n R_{s,\alpha,\pi_1}\n^p_{\gamma(L^2(0,s;H),E)}<\infty.
\end{align*}
  \end{proof}

\subsection{Examples}

\subsubsection{Stochastic transport equation with delay.}
Let $E=\tilde{E}=C_0([0,1])=\{ f\in C([0,1]):f(0)=0 \}$. Consider the following stochastic transport equation with delay in $C_0([0,1])$:
\begin{align}
\left\{ \begin{array}{l} dy(t,\xi)=( -\frac{\partial y(t,\xi)}{\partial \xi}-\mu y(t,\xi)) dt +\Big[\int_{t-1}^{t}\varphi(s-t,\xi)y(s,\xi)ds\\\quad\quad +f_1(y(t,\xi))+\int_{t-1}^{t}k(s-t,\xi)f_2(y(s,\xi))ds\Big]dt+\psi(\xi)dW(t),\quad t\geq 0;\\
\frac{\partial y(t,0)}{\partial \xi}=0, \quad y(t,0)=0;\\
y(0,\xi)=x_0(\xi), \quad y(\theta,\xi)=f_0(\theta,\xi), \quad \theta \in[-1,0], \xi\in [0,1];\label{ex_transp}
 \end{array}\right. 
\end{align}  
for the initial conditions $x_0\in C_0([0,1])$, $f_0\in L^p(-1,0;C_0([0,1]))$ and where $p\geq 1$, $\varphi,k\in C([0,1];L^{p'}(-1,0))$ for $p'\in(1,\infty]$ such that $\inv{p}+\inv{p'}=1$, and $f_1,f_2:\R\to\R$ are Lipschitz functions, and 
  $\psi\in C_0([0,1])$,  $W$ is one-dimensional Brownian motion. 
Let $B$ be a differential operator on $E=C_0(0,1)$ such that
\begin{align*}
Bx=-\frac{dx}{d\xi}-\mu x, \ D(B)=\{x\in C^{1}([0,1]):x(0)=x'(0)=0\}.
\end{align*}
 By \cite{EngNag} (see p. 86 and section 5.11)  it follows that  $(B, D(B))$ generates strongly continuous nilpotent semigroup  $(S(t))_{t\geq 0}$ on $C_0([0,1])$ such that 
\begin{align}\label{r:S}
(S(t)x)(\xi)=\eqtwo{e^{-\xi\mu }x(\xi-t)}{\xi-t\geq 0}{0}{\xi-t<0}, 
\end{align}
for all $x\in C_0([0,1])$ and all $s\in[0,1]$.
Let us introduce the notation:
\begin{align}\label{notation_transp_1}
\phi(x,h)(\xi)&=\int_{-1}^{0}\varphi(\theta,\xi)h(\theta,\xi)d\theta+f_1(x(\xi))+ \int_{-1}^{0}k(\theta,\xi)f_2(h(\theta,\xi))d\theta,\\
(\psi u)(\xi)&=\psi(\xi)u
\end{align}
for all $[x,h]'\in\Ep{p}=C_0([0,1])\times L^p(-1,0;C_0([0,1]))$, and all $u\in \R$.
Then, we can rewrite \eqref{ex_transp} in the form \eqref{SDEa}. Observe that $\phi$ is Lipschitz-continuous with the Lipschitz constant $$L=2^{\inv{p}}\lee   L_{f_{1}} \maxsym (L_{f_{2}}\len k \pn_{C([0,1];L^{p'}(-1,0;E))}+\len \varphi \pn_{C([0,1];L^{p'}(-1,0;E))})\p,$$ where $L_{f_{1}}, L_{f_{2}}$ are the Lipschitz constants of $f_1, f_2$, respectively. Hence, since $\len S(s)\pn_{\Ca{L}(C_0([0,1])}\leq 1$ for all $s>0$, $\phi$ satisfies \ref{hphi} with $a(t)=L$. \par
 Now we show that the assumption \ref{hpsi} holds.
We prove that $[0,t]\ni s \mapsto S(s)\psi \in C_0([0,1])$ represents an operator $R_{\psi,t}$ in $\gamma(L^2(0,t),C_0[0,1])$ defined as $R_{\psi,t}f=\int_0^tS(s)\psi f(s)ds$. For  $t=1$ we have
\begin{align*}
R_{\psi,1}f(\xi)=&\lee\int_0^1S(s)\psi f(s)ds\p(\xi)=\int_0^1 e^{-\xi\mu} 1_{[s,1]}(\xi)\psi (\xi-s)f(s)ds\\&=e^{-\xi\mu}\int_0^{\xi}  \psi (\xi-s)f(s)ds 
\end{align*}
and then for all $\xi\in[0,1]$
\begin{align}\label{Rgamma}
|R_{\psi,1}f(\xi)|\leq \len \psi\pn_{\infty} \int_0^{\xi} f(t)dt. 
\end{align}
Let $h_0(\xi)=1, h_{k}(\xi)=2^{\inv{2}(n-1)}\lee 1_{(\frac{2j-2}{2^{n}},\frac{2j-1}{2^{n}})}(\xi)-1_{(\frac{2j-1}{2^{n}},\frac{2j}{2^{n}})}(\xi)\p$ for all $\xi\in[0,1]$ and $k=2^{n-1}+j-1$ with $n=1,2\ldots$, $j=1,2,\ldots, 2^{n-1}$ be the Haar basis on $L^2(0,1)$. Then using \eqref{Rgamma} we obtain, for all $\xi \in [0,1]$ and $k\geq 1$,
\begin{align}\label{Rgamma0}
|R_{\psi,1}h_{k}(\xi)|&\leq \len \psi\pn_{\infty} \int_0^1 h_{k}(t) dt\leq 1_{(\frac{2j-2}{2^{n}},\frac{2j}{2^{n}})}(\xi)\len \psi\pn_{\infty}2^{\inv{2}(n-1)}\inv{2^{n}},\\ \quad \text{where } k&=2^{n-1}+j-1,\ n=1,2,\ldots, j=1,\ldots,2^{n-1}.\notag
\end{align}
Let $\{\gamma_{k}:k=0 \text{ or }k=2^{n-1}+j-1,n=1,2\ldots, j=1,2,\ldots, 2^{n-1} \}$ be a Gaussian sequence. Then for every $\xi\in [0,1] $ and all $\beta>1$ and sufficiently large $1<N< M$ such that $N=2^{n_N-1}+j_N-1$ and $M=2^{n_M-1}+j_M-1$ for some $1\leq j_N\leq 2^{n_N-1}$, $1\leq j_M\leq 2^{n_M-1}$ and $n_N,n_M\geq 1$ we have, almost surely,
\begin{align}\label{Rgamma2}
\sum_{k=N}^M|\gamma_{k}R_{\psi,1}h_{k}(\xi)|&\leq\sum_{k=N}^M\sqrt{2\beta\log(k+1)}|R_{\psi,1}h_{k}(\xi)|\\ \notag&\leq \len \psi\pn_{\infty}\sum_{n=n_N}^{m_M}\sum_{j=1}^{2^{n-1}}1_{(\frac{2j-2}{2^{n}},\frac{2j}{2^{n}})}(\xi)\sqrt{2\beta\log(j+2^{n-1})}2^{-\inv{2}n-\inv{2}}\\
&\leq
\len\psi\pn_{\infty}\sum_{n=n_N}^{m_M}\sqrt{2\beta\log(j'+2^{n-1})}2^{-\inv{2}n-\inv{2}}
\notag\\
&\leq\len\psi\pn_{\infty}\sqrt{2\beta\log 2}\sum_{n=n_N}^{m_M}2^{\inv{2}+\inv{4}n}2^{-\inv{2}n-\inv{2}}
\notag\\&=\len\psi\pn_{\infty}\sqrt{2\beta\log 2}\sum_{n=n_N}^{m_M}\lee\inv{\sqrt[4]{2}}\p^{n},\notag
\end{align}
where  we use \eqref{Rgamma0} and the following property of Gaussian sequences: for every $\beta>1$ the events $|\gamma_k|\leq \sqrt{2\beta\log(k+1)}$ hold for all but finitely many $k$ and the inequalities: $\log (2^{n-1}+j')\leq \log 2^n=n\log 2$ and $\sqrt{n}\leq 2^{\inv{2}+\inv{4}n}$ hold.
For all $N=2^{n-1}+j-1$ with $n=1,2,\ldots$ and $1 \leq j\leq 2^{n-1}$ let $S_N(\xi)=\sum_{k=1}^N\gamma_{k}R_{\psi,1}h_{k}(\xi)$, $\xi \in [0,1]$. 
Hence the sequence $(S_N)_{N\geq 1}$   converges to $Y$, almost surely, absolutely and uniformly for
all $\xi\in[0,1]$.  Since each $\xi\mapsto S_N(\xi)$ is continuous, it implies that the function $\xi\mapsto Y(\xi)$ belongs to $C[0,1]$. Moreover, in the same way as in \eqref{Rgamma2} we obtain:  
\begin{align}\label{Rgamma3}
\E\ \len Y \pn^2_\infty &=
\E\  \sup_{\xi\in[0,1]}\left|\sum_{k=1}^\infty\gamma_{k}R_{\psi,1}h_{k}(\xi)\right|^2\\&\leq   \len\psi\pn^2_{\infty}\lee\sum_{n=1}^{N-1}2^{-n-1}+ 2\beta\log 2\lee\sum_{n=N}^{\infty}\lee\inv{\sqrt[4]{2}}\p^{n}\p^2\p<\infty, \notag
\end{align}
where $N>1$ is sufficiently large.
 By \eqref{Rgamma2}-\eqref{Rgamma3} and the Ito-Nisio theorem (see Proposition 2.11 in \cite{DaPratoZabczyk})  the sequence $\lee S_n\p_{n\geq 0}$ is converged in $L^2(\Omega;C([0,1]))$ and a.s. to $Y\in L^2(\Omega;C([0,1]))$.
  
 Therefore,  $R_{\psi,1}\in\gamma(L^2(0,1);C_0([0,1]))$ and form Corollary 7.2 in  \cite{vanNeervenVeraarWeis2} it follows that $R_{\psi,t}\in\gamma(L^2(0,t);C_0([0,1]))$ for all $t>0$.
   Finally, by Proposition \ref{p:} we have the existence and uniqueness  of a weak solution to \eqref{ex_transp} in the spaces $\Bo{SL}_{\Ca{F}}^{p\maxsym q}(0,t;E)$ for every $q\geq 1$ and the weak solution satisfies 
\begin{align*}
X(t,\xi) &= 1_{[0,\infty)}(\xi-t)e^{-\xi\mu }x(\xi-t)+e^{-\xi\mu }\int_{0}^{t-\xi} \phi(X(s),X_s)(\xi-t+s)ds\\&\quad+e^{-\xi\mu}\int_0^{t\minsym \xi}\psi(\xi-s)dW(s),
\end{align*}
for all $t>0$ and $\xi \in[0,1]$.
  
\subsubsection{Stochastic McKendrick equation with delay.}
Let $E=\tilde{E}=L^1(\Ca{O})$, where $\Ca{O}=(0,\infty)$. Consider the following stochastic delay McKendrick equation in $L^1(\Ca{O})$:
\begin{align}\label{ex_McKendrick}
\left\{ \begin{array}{l} dy(t,\xi)=( -\frac{\partial y(t,\xi)}{\partial \xi}-\mu(\xi) y(t,\xi)) dt +\Big[\int_{t-1}^{t}\varphi(s-t,\xi)y(s,\xi)ds\\\quad\quad+f_1(y(t,\xi))+ \int_{t-1}^{t}k(s-t,\xi)f_2(y(s,\xi))ds\Big]dt+\psi(\xi)dW(t),\quad t\geq 0;\\
y(t,0)=\int_0^\infty b(a)y(t,a)ds;\\
y(0,\xi)=x_0(\xi), \quad y(\theta,\xi)=f_0(\theta,\xi), \quad \theta\in[-1,0], \xi\in[0,1]; 
 \end{array}\right. 
\end{align}  
for the initial conditions $x_0\in L^1(\Ca{O})$, $f_0\in L^p(-1,0;L^1(\Ca{O}))$, where $p\geq 1$, $\mu, b \in L^\infty(\Ca{O})$, $\varphi,k\in L^\infty(\Ca{O};L^{p'}(-1,0))$ for some $p'\in(1,\infty]$ such that $\inv{p}+\inv{p'}=1$, and $f_1, f_2:\R\to\R$ are Lipschitz functions, and $W$ is one-dimensional Brownian motion; $\psi\in\Ca{L}(\R{},L^1(\Ca{O}))$, $\psi h(a)=h\sigma(a)$ for all $h\in\R{}$ for some  $\sigma\in L^1(\Ca{O})$ such that $\operatorname{supp}\sigma\subset[0,d]$ $(d\in\R{}_+)$ and $\sigma\in L^2(0,d)$.

  Let $B$ be a linear operator on $L^1(\Ca{O})$ such that:
\begin{align*}
D(B)={\rm Ker} (K),\ Bg=-\frac{d}{da}g-\mu g,
\end{align*}
where  $K:W^{1,1}(\Ca{O})\to \R{}$, $Kg=g(0)-\int_{\Ca{O}}b(a)g(a)da$. Then, by Theorem 2 in~\cite{bobrowski2010} it follows that $B$ generates the McKendrick semigroup $\left(S(t)\right)_{t\geq 0}$. Hence for all $t\geq 0$
\begin{align}\label{pó³grMc}
S(t)g(a)=e^{-\int_{a-t}^a\mu(r)dr}\tilde{g}(a-t), \ a\geq 0,
\end{align} 
where $\tilde{g}(a)=g(a)$, $\tilde{g}(-a)=g_2(a)$, $a\geq 0$ and $g_2$ belongs to the weighted Banach space $$L_{w}^1(\Ca{O})=\left\{g:\Ca{O}\to\R{}:\n g\n_{L^1_w(\Ca{O})}=\int_{\Ca{O}}|g(a)|e^{-wa}da<\infty\right\}$$ for $w>\n b_\mu\n_\infty$ and satisfies, almost everywhere, the equation
\begin{align}\label{popsplot}
g_2=b_\mu\star g_2+T_{\mu,b}g,  
\end{align} 
where 
\begin{align*}
&T_{\mu,b}:L^1(\Ca{O})\to L_w^1(\Ca{O}),\ T_{\mu,b}g(s)=\int_{s}^\infty e^{-\int_{a-s}^a\mu(r)dr}g(a-s)b(a)da, \\
&b_{\mu}(s)=e^{-\int_0^s\mu(r)dr}b(s), \ s\geq 0,
\end{align*}
 and $\star$ denotes the convolution operation in $L^1(\Ca{O})$. For every $g\in L^1(\Ca{O})$ the function $\tilde{g}=(g,g_2)\in L^1(\Ca{O})\times L_w^1(\Ca{O})$ defined by \eqref{popsplot} is called $(\mu,b)$-extension of $g$. \par
Let us denote by $\Ep{p}=L^1(\Ca{O})\times L^p(-1,0;L^1(\Ca{O}))$ the state space for the delay equation \eqref{ex_McKendrick}. Let $\phi:\Ep{p}\to L^1(\Ca{O})$ be given by \eqref{notation_transp_1}. It is easy to show that $\phi$ is Lipschitz-continuous with the Lipschitz constant $$L=2^{\inv{p}}\lee   L_{f_{1}} \maxsym (L_{f_{2}}\len k \pn_{L^\infty(\Ca{O};L^{p'}(-1,0;E))}+\len \varphi \pn_{L^\infty(\Ca{O};L^{p'}(-1,0;E))})\p,$$ where $L_{f_{1}}, L_{f_{2}}$ are the Lipschitz constants of $f_1, f_2$, respectively. Hence $\phi$ satisfies \ref{hphi} with $a(t)=LS(t)$. We show in Proposition \ref{p:psi_McKen} that the assumptions \ref{hpsi} and \eqref{hpsi_alpha_1}-\eqref{hpsi_alpha_2} hold. Therefore, we can rewrite \eqref{ex_McKendrick} in the form \eqref{SDEa} and apply Proposition \ref{p:} to prove existence, uniqueness and continuity of a weak solution to \eqref{ex_McKendrick}.
 
\begin{prop}\label{p:psi_McKen}  Consider \eqref{ex_McKendrick}. Then, the operator $\psi\in\Ca{L}(\R{},L^1(\Ca{O}))$ defined by $\psi h(a)=h\sigma(a)$ for all $h\in\R{}$ and for some  $\sigma\in L^1(\Ca{O})$ such that $\operatorname{supp}\sigma\subset[0,d]$ $(d\in\R{}_+)$ and $\sigma\in L^2(0,d)$ satisfies    
\ref{hpsi} and \eqref{hpsi_alpha_1}-\eqref{hpsi_alpha_2}.
\end{prop}
\begin{proof}
The $\gamma$-Fubini isomorphism (see Proposition 2.6 in \cite{vanNeervenVeraarWeis}) between the Banach spaces  $L^1(\Ca{O};(L^2(0,t))^*)$ and $\gamma(L^2(0,t);L^1(\Ca{O}))$ implies that to prove condition \ref{hpsi} it is enough to find $t>0$ such that
\begin{align}\label{r:istSRPzSWO}
\int_{\Ca{O}}\sup_{\n f\n_{L^2(0,t)}\leq 1}\left|\lee\int_0^t S(s)\psi f(s)ds\p(a)\right|da<\infty.
\end{align}
We take $t=d$, then for a.e. $a\geq 0$ by the Cauchy-Schwarz inequality it follows that
\begin{align*}
\left|\lee\int_0^d S(s)\psi f(s)ds\p(a)\right|\leq \n f\n_{L^2(0,d)} \lee\int_0^de^{-2\int_{a-s}^a\mu(r)dr}\tilde{\sigma}^2(a-s)ds\p^{\inv{2}}.
\end{align*}      
Hence
\begin{align}\notag
\int_{\Ca{O}}\sup_{\n f\n_{L^2(0,d)}\leq 1}&\left|\int_0^d S(s)\psi f(s)ds)(a)\right|da\\&\leq \int_\Ca{O}\lee\int_0^de^{-2\int_{a-s}^a\mu(r)dr}\tilde{\sigma}^2(a-s)ds\p^{\inv{2}}da\notag\\
\label{r:istSRPzSWO2}&\leq \int_0^d \lee\int_{0}^d e^{-2\int_{a-s}^a\mu(r)dr}\tilde{\sigma}^2(a-s)ds\p^\inv{2}da\\&\quad+ \int_d^\infty \lee\int_{0}^d e^{-2\int_{a-s}^a\mu(r)dr}\sigma^2(a-s)ds\p^\inv{2}da.\notag
\end{align}
Since $\sigma\in L^2(0,d)$, we can estimate the second integral on right hand side of \eqref{r:istSRPzSWO2} as follows
\begin{align*}
 \int_d^\infty &\lee\int_{0}^d e^{-2\int_{a-s}^a\mu(r)dr}\sigma^2(a-s)ds\p^\inv{2}da\\&=\int_d^{2d} \lee\int_{a-d}^d e^{-2\int_{s}^a\mu(r)dr}\sigma^2(s)ds\p^\inv{2}da\leq d \n\sigma\n_{L^2(0,d)}.
\end{align*}
For the first integral on right hand side of \eqref{r:istSRPzSWO2} we obtain
\begin{align*}
\int_0^d &\lee\int_{0}^d e^{-2\int_{a-s}^a\mu(r)dr}\tilde{\sigma}^2(a-s)ds\p^\inv{2}da\\&\leq
\int_0^d \lee\int_{0}^a e^{-2\int_{a-s}^a\mu(r)dr}\sigma^2(a-s)ds\p^\inv{2}da\\&\quad+\int_0^d \lee\int_{a}^d e^{-2\int_{0}^a\mu(r)dr}\sigma_2^2(s-a)ds\p^\inv{2}da\\
&\leq d \n\sigma\n_{L^2(0,d)}+d\n\sigma_2\n_{L^2(0,d)}.
\end{align*}
 We show that $\sigma_2\in L^2(0,d)$. 
  Recall that $\sigma_2\in L_w^1(\Ca{O})$ for $w>\n b_\mu\n_{\infty}$ is the solution to (see \eqref{popsplot}):
\begin{align}\label{sigma2}
\sigma_2=b_\mu\star\sigma_2 + T_{\mu, b}\sigma.
\end{align}
We denote the restriction of $T_{\mu,b}$ to $L^2(0,d)$ by $T_{\mu, b, 2}$. Then, for all $s\in[0,t]$ and $g\in L^2(0,d)$  
\begin{align}\label{C}
&T_{\mu,b,2}g(s)=\int_0^de^{-\int_{a}^{a+s}\mu(r)dr}b(a+s)g(a)da,\\
&\norm{ T_{\mu,b,2}g}_{L_w^2(0,d)}\leq \sqrt{d}\n b\n_\infty\n g\n_{L^2(0,d)},\notag
\end{align}
and hence $T_{\mu, b,2}\in\Ca{L}(L^2(0,d),L_w^2(0,d))$. Moreover, by
the Young inequality for convolutions it follows that for all $g\in L_w^2(0,d)$ we have
\begin{align*}
\n b_\mu\star g\n_{L_w^2(0,d)}\leq \n b_\mu\n_{L_w^1(0,d)} \n g\n_{L_w^2(0,d)}\leq\frac{\n b_\mu\n_\infty }{w}\n g\n_{L_w^2(0,d)},
\end{align*}
Hence for $w>\n b_\mu\n_\infty$ the mapping $L_w^2(0,d)\ni g\mapsto b_\mu\star g + T_{\mu, b,2}\sigma\in L_w^2(0,d)$ is a strict contraction, thus by the Banach fixed-point theorem there exists in $L_w^2(0,d)$ a unique solution to \eqref{sigma2}. It is clear that $L_w^2(0,d)$ and $L^2(0,d)$ are isomorphic, thus  $\sigma_2\in L^2(0,d)$. Therefore, the proof of \eqref{r:istSRPzSWO} is complete.\par

Let $\alpha\in(\inv{q\maxsym p},\inv{2})$ and $t>0$. Now we prove that
\begin{align*}
&\sup_{s\in[0,t]}\n u\mapsto (s-u)^{-\alpha}S(s-u)\psi\n_{\gamma(L^2(0,s),L^1(\Ca{O}))}<\infty.
\end{align*}
Fix $s\in[0,t]$. Notice that by $\gamma$-Fubini isomorphism we have
\begin{align*}
&\norm{ u\mapsto (s-u)^{-\alpha}S(s-u)\psi}_{\gamma(L^2(0,s),L^1(\Ca{O}))}\\&\hspace{2cm}\leq C_\gamma \int_{\Ca{O}}\sup_{\n f\n_{L^2(0,s)}\leq 1}\left|\lee\int_0^s (s-u)^{-\alpha}S(s-u)\psi f(u)du\p(a)\right|da,
\end{align*}
for some constant $C_\gamma>0$.
By the Cauchy-Schwarz inequality  
\begin{align}\label{cPop}
&\int_{\Ca{O}}\sup_{\n f\n_{L^2(0,s)}\leq 1}\left|\lee\int_0^s (s-u)^{-\alpha}S(s-u)\psi f(u)du\p(a)\right|da\\
& \notag\hspace{2cm}\leq \int_0^s \lee\int_{0}^su^{-2\alpha} e^{-2\int_{a-u}^a\mu(r)dr}\tilde{\sigma}^2(a-u)ds\p^\inv{2}da \\&\hspace{3cm}+ \int_s^\infty \lee\int_{0}^s u^{-2\alpha}e^{-2\int_{a-u}^a\mu(r)dr}\sigma^2(a-u)ds\p^\inv{2}da.\notag
\end{align}
Using the assumption $\sigma\in L^2(0,d)$, the Cauchy-Schwarz inequality and Fubini's theorem we can estimate the second integral on the right hand side of \eqref{cPop} as follows 
\begin{align*}
&\int_s^\infty \lee\int_{0}^s u^{-2\alpha}e^{-2\int_{a-u}^a\mu(r)dr}\sigma^2(a-u)ds\p^\inv{2}da\\
&\hspace{3cm}=\int_s^{s+d} \lee\int_{a-s}^{d\minsym a} (a-u)^{-2\alpha}e^{-2\int_{u}^a\mu(r)dr}\sigma^2(u)ds\p^\inv{2}da\\
&\hspace{3cm}\leq \sqrt{d}\lee\int_s^{s+d} \int_{a-s}^{d\minsym a} (a-u)^{-2\alpha}\sigma^2(u)dsda\p^\inv{2}\\
&\hspace{3cm}=\sqrt{d}\lee\int_0^{d} \sigma^2(u)\int_{s\maxsym u}^{u+s} (a-u)^{-2\alpha}dads\p^\inv{2}\\&\hspace{3cm}\leq \sqrt{d\frac{s^{1-2\alpha}-1}{1-2\alpha}}\n\sigma\n_{L^2(0,d)}.
\end{align*}
We decompose the first term on the right hand side of \eqref{cPop} as
\begin{align*}
&\int_0^s \lee\int_{0}^su^{-2\alpha} e^{-2\int_{a-u}^a\mu(r)dr}\tilde{\sigma}^2(a-u)ds\p^\inv{2}da\\&\hspace{3cm}=\int_0^s \lee\int_{0}^au^{-2\alpha} e^{-2\int_{a-u}^a\mu(r)dr}\sigma^2(a-u)ds\p^\inv{2}da\\&\hspace{3.5cm}+\int_0^s \lee\int_{a}^su^{-2\alpha} e^{-2\int_{a-u}^a\mu(r)dr}\sigma_2^2(a-u)ds\p^\inv{2}da.
\end{align*}
The Cauchy-Schwarz inquality and Fubini's theorem gives
\begin{align*}
&\int_0^s \lee\int_{0}^au^{-2\alpha} e^{-2\int_{a-u}^a\mu(r)dr}\sigma^2(a-u)ds\p^\inv{2}da\\&\hspace{1cm}\leq\sqrt{s}\lee\int_0^su^{-2\alpha}\int_0^{s-u}\sigma^2(a)dadu\p^{\inv{2}}\leq\sqrt{\frac{s(s^{1-2\alpha}-1)}{1-2\alpha}}\n\sigma\n_{L^2(0,d)}.
\end{align*}
and
\begin{align*} 
&\int_0^s \lee\int_{a}^su^{-2\alpha} e^{-2\int_{a-u}^a\mu(r)dr}\sigma_2^2(a-u)ds\p^\inv{2}da\\&\hspace{1cm}\leq\sqrt{s}\Big(\int_0^su^{-2\alpha}\int_0^{u}\sigma_2^2(a)dadu\Big)^{\inv{2}}\leq\sqrt{\frac{s(s^{1-2\alpha}-1)}{1-2\alpha}}\n\sigma_2\n_{L^2(0,s)}.
\end{align*}
Similarly as in the first part of the proof we obtain that $\sigma_2\in L^2(0,s)$ for all $s>0$. Therefore,
\begin{align*}
&\sup_{s\in[0,t]}\n u\mapsto (s-u)^{-\alpha}S(s-u)\psi\n_{\gamma(L^2(0,s),L^1(\Ca{O}))}\\&\hspace{3cm}\leq C_\gamma \sqrt{(d\maxsym t)\frac{t^{1-2\alpha}-1}{1-2\alpha}}(2\n\sigma\n_{L^2(0,d)}+\n\sigma_2\n_{L^2(0,t)})<\infty.
\end{align*}
\end{proof}

\section*{Acknowledgments}
The author wishes to thank Professor Anna Chojnowska-Michalik and the referee for many helpful suggestions and comments.


\begin{thebibliography}{1}

\bibitem{Baker} C.T.H. Baker, G.A. Bocharov, and F.A. Rihan. \emph{A report on the use of delay differential
equations in numerical modelling in the biosciences}, Numerical Analysis Report 343 (1999), 1–46.

\bibitem{BarucciGozzi1999}
E. Barucci, F. Gozzi,  \emph{Optimal advertising with a continuum of goods} Annals of Operations Research 88 (1999), 15--29.

\bibitem{Batkai2005}
A. B{\'a}tkai, S. Piazzera \emph{Semigroups for delay equations}, Research
  Notes in Mathematics 10,
A K Peters Ltd., Wellesley, MA 2005.

\bibitem{bobrowski2010}
 A. Bobrowski, \emph{Lord {K}elvin's method of images in semigroup theory}, Semigroup Forum {81}(3) (2010), 435--445.

\bibitem{BrzezniakNeerven_stochconv}
Z. Brze{\'z}niak, J.M.A.M. van Neerven \emph{ Stochastic convolution in separable
  {B}anach spaces and the stochastic linear {C}auchy problem}, Studia Mathematica {143}(1) (2000), 43--74.

\bibitem{chojnowskaMichalik_SDEHilbert}
A. Chojnowska-Michalik, \emph{Stochastic differential equations in {H}ilbert spaces}
in: Probability theory ({P}apers, {VII}th {S}emester, {S}tefan
  {B}anach {I}nternat. {M}ath. {C}enter, {W}arsaw, 1976), Banach Center
  Publ. 5, (1979), 53--74.

\bibitem{coxgorajski}
S. Cox, M. G{\'o}rajski,  \emph{Vector-valued stochastic delay equations-a semigroup
  approach}, Semigroup Forum {82} (2011), 389--411.

\bibitem{coxVeraar}
S. Cox, M. Veraar,
\emph{Vector-valued decoupling and the {B}urkholder--{D}avis--{G}undy inequality}
Illinois Journal of Mathematics,
  {55}(1) (2011), 343--375.
  
\bibitem{Crewe}
P. Crewe, \emph{Infinitely delayed stochastic evolution equations in {UMD} Banach spaces} arXiv:1011.2615v1

\bibitem{DaKwaZab87}
 G. Da~Prato, S. Kwapie{\'n}, J. Zabczyk, \emph{Regularity of solutions of linear
  stochastic equations in {H}ilbert spaces} Stochastics {23}(1) (1987), 1--23. 

\bibitem{DaPratoZabczyk}
G. Da~Prato, J. Zabczyk, \emph{Stochastic equations in infinite dimensions},
  Encyclopedia of Mathematics and its Applications 44, Cambridge University Press, Cambridge, 1992.

\bibitem{Denk}
R. Denk, M. Hieber, and J. Pr\"uss, \emph{R-boundedness, Fourier multipliers and
problems of elliptic and parabolic type}, Mem. Amer. Math. Soc. {166}(788) (2003), viii+114.

\bibitem{EngNag}
K. Engel, R. Nagel, \emph{One-parameter semigroups for linear evolution equations},
  Graduate Texts in Mathematics 194,
Springer-Verlag, New York, 2000.

\bibitem{Erneux}
T. Erneux, \emph{Applied Delay Differential
Equations},
  Surveys and Tutorials in the Applied
Mathematical Sciences 3,
Springer, New York, 2009.
 
\bibitem{Garling_rmte}
D. Garling \emph{Random martingale transform inequalities},
in: Probability in {B}anach spaces 6, {S}andbjerg, 1986, Progr. Probab. {20} Birkh\"auser Boston, Boston, MA
  (1990), 101--119.

\bibitem{GorajskiSDE}
M. G{\'o}rajski, \emph{Vector-valued stochastic delay equations - a weak solution
  and its {M}arkovian representation}, arXiv preprint
	arXiv:1301.5300, (2013), submitted for publication.

\bibitem{Kalton}
N. J. Kalton,  and L. Weis, \emph{The $H^{\infty}$-functional calculus and square function estimates}, Manuscript in preparation.


\bibitem{Kunze}
M. C. Kunze, \emph{Martingale problems on {B}anach spaces}, arXiv preprint arXiv:1009.2650, (2012).

\bibitem{Mao}  X. Mao, \emph{Stochastic {D}ifferential {E}quations and {T}heir {A}pplications}, Horwood Publishing Series in
Mathematics \& Applications, Horwood, Chichester 1997.

\bibitem{Mohammed}
S.E.A. Mohammed, \emph{Stochastic Functional Differential Equations}, Research Notes in Mathematics,
vol. 99. Pitman, Boston 1984.

\bibitem{vanNeervenadjoint}
J.M.A.M. van Neerven, \emph{The adjoint of a semigroup of linear operators},
  Lecture Notes in Mathematics, 1529, Springer-Verlag, Berlin, 1992.

\bibitem{vanNeervenISEM}
J.M.A.M. van Neerven, \emph{Stochastic Evolution Equations}, Notes to the 11th
  Internet Seminar, 2008.

\bibitem{vanNeervenVeraarWeis_SEEinUMD}
J.M.A.M. van Neerven, M. Veraar, L. Weis, \emph{Stochastic evolution equations in
  {UMD} {B}anach spaces},
 J. Funct. Anal. {255} (2008), 940--993.

\bibitem{vanNeervenVeraar_Fub}
J.M.A.M. van Neerven, M.C. Veraar  \emph{On the stochastic {F}ubini theorem in
  infinite dimensions},
in: Stochastic partial differential equations and
  applications---{VII}, Lect. Notes Pure Appl. Math. vol. 245, Chapman \& Hall/CRC, Boca Raton, FL 2006, 323--336.

\bibitem{vanNeervenVeraarWeis}
J.M.A.M. van Neerven, M.C. Veraar , L. Weis \emph{Stochastic integration in {UMD}
  {B}anach spaces}, Annals of Probability {35}(4) (2007), 1438--1478. 

\bibitem{vanNeervenVeraarWeis2}
 J.M.A.M. van Neerven and L. Weis, \emph{Stochastic integration of functions with
  values in a {B}anach space},
Studia Mathematica {166}(2) (2005), 131--170.

\bibitem{PeszZab}
S. Peszat and J. Zabczyk, \emph{Stochastic partial differential equations with
  {L}\'evy noise}, \emph{Encyclopedia of Mathematics and its Applications}
  113, Cambridge University Press, Cambridge, 2007.

\bibitem{Qi-Min}
Z. Qi-Min, L. Wen-An, N. Zan-Kan \emph{Existence, uniqueness and exponential stability for stochastic age-dependent population}, Applied Mathematics and Computation {154}, (2004), 183--201.


\bibitem{veraar_thesis}
M. Veraar, \emph{Stochastic integration in {B}anach spaces and applications to
  parabolic evolution equations}, ISBN 978-90-9021380-4, 2006.

\bibitem{VeraarZimmerschied}
M. Veraar and J. Zimmerschied, \emph{Non-autonomous stochastic {C}auchy problems in
  {B}anach spaces}, Studia Mathematica {185}(1) (2008), 1--34. 

\bibitem{Webb1985}
G. Webb,\emph{Theory of nonlinear age-dependent population dynamics}, 89, CRC, 1985.


\end{thebibliography}
\end{document}